\documentclass[onefignum,onetabnum]{siamart220329}

\usepackage{graphicx}
\usepackage{multirow}
\usepackage{amsmath,amssymb,amsfonts}
\usepackage{mathrsfs}
\usepackage[title]{appendix}
\usepackage{xcolor}
\usepackage{textcomp}
\usepackage{manyfoot}
\usepackage{booktabs}
\usepackage{algorithmicx,algorithm,algpseudocode}
\usepackage{ dsfont }
\usepackage{caption}

\usepackage{latexsym,enumitem,verbatim,amsfonts,microtype}
\usepackage{color}
\usepackage{graphicx}
\usepackage[active]{srcltx}
\usepackage{cite}
\usepackage{cases}
\usepackage{mathtools}
\usepackage[figuresright]{rotating}
\usepackage{microtype}
\usepackage[colorinlistoftodos,prependcaption,textsize=footnotesize]{todonotes}
\usepackage[latin1]{inputenc}
\usepackage[framemethod=tikz]{mdframed}
\mdfsetup{%
	skipbelow=4pt,
	skipabove=8pt,
	linewidth=1.25pt,
	backgroundcolor=gray!10,
	userdefinedwidth=\textwidth,
	roundcorner=10pt,
}

\allowdisplaybreaks[2]

\def\R{{\mathbb{R}}}
\def\B{{\mathbb{B}}}

\def\argmin{\mathop{\rm arg\,min}}

\def\bd{{\partial}}
\def\spb{{{\cal SPB}(\R^d)}}

\numberwithin{equation}{section}

\ifpdf
  \DeclareGraphicsExtensions{.eps,.pdf,.png,.jpg}
\else
  \DeclareGraphicsExtensions{.eps}
\fi

\newsiamremark{remark}{Remark}
\newsiamremark{hypothesis}{Hypothesis}
\crefname{hypothesis}{Hypothesis}{Hypotheses}
\newsiamthm{claim}{Claim}

\newtheorem{example}{Example}[section]

\headers{SPB functions and Gaussian smoothing}{M. Lei, T. K. Pong, S. Sun and M.-C. Yue}

\title{Subdifferentially polynomially bounded functions and Gaussian smoothing-based zeroth-order optimization
\thanks{Submitted to the editors \today.
\funding{Ting Kei Pong is supported in part by a Research Scheme of the Research Grants Council of Hong Kong SAR, China (project T22-504/21R). Shuqin Sun is supported in part by the Opening Project of Sichuan Province University Key Laboratory of Bridge Non-destruction Detecting and Engineering Computing (project 2023QYJ08). Man-Chung Yue is supported in part by the Hong Kong Research Grants Council under the GRF project 15304422.}}}

\author{Ming Lei\thanks{College of Applied Mathematics, Chengdu University of Information Technology, Chengdu, People's Republic of China
  (\email{leim@cuit.edu.cn}).}
\and   Ting Kei Pong\thanks{Department of Applied Mathematics, The Hong Kong Polytechnic University, Hong Kong, People's Republic of China
  (\email{tk.pong@polyu.edu.hk}).}
\and Shuqin Sun\thanks{Key Laboratory of Optimization Theory and Applications at China West Normal University of Sichuan Province, School of Mathematics Education, China West Normal University, Nanchong, People's Republic of China
  (\email{sunshuqinsusan@163.com}).}
\and Man-Chung Yue\thanks{Musketeers Foundation Institute of Data Science and Department of Data and Systems Engineering, The University of Hong Kong, Hong Kong, People's Republic of China
  (\email{mcyue@hku.hk}).}
}

\usepackage{amsopn}

\definecolor{revise}{rgb}{0,0,0}
\definecolor{rerevise}{rgb}{0,0,0}

\ifpdf
\hypersetup{
}
\fi

\begin{document}

\maketitle

\begin{abstract}
We {\color{revise}study} the class of subdifferentially polynomially bounded (SPB) functions, which is a rich class of locally Lipschitz functions that encompasses all Lipschitz functions, all gradient- or Hessian-Lipschitz functions, and even some non-smooth locally Lipschitz functions.
We show that SPB functions are compatible with Gaussian smoothing (GS), in the sense that the GS of any SPB function is well-defined and satisfies a descent lemma akin to gradient-Lipschitz functions, with the Lipschitz constant replaced by a polynomial function.
Leveraging this descent lemma, we propose GS-based zeroth-order optimization algorithms with an adaptive stepsize strategy for {\color{revise}minimizing} SPB functions, and analyze their {\color{revise}convergence rates with respect to both relative and absolute stationarity measures. Finally, we also establish the iteration complexity for achieving a $(\delta, \epsilon)$-approximate stationary point, based on a novel quantification of Goldstein stationarity via the GS gradient that could be of independent interest.}
\end{abstract}



\section{Introduction}	

Zeroth-order optimization (a.k.a. derivative-free optimization) refers to optimization problems where the objective function can be accessed only through a zeroth-order oracle, a routine for evaluating the function at a prescribed point.
Zeroth-order optimization {\color{revise}often arises in situations where one aims at optimally exploring or configuring physical environments using experimental data or computer simulations}, and has attracted intense research over the last few decades.
We refer the readers to the expositions \cite{ConnSchVic09, Larson2019} and references therein for classic works and recent developments on zeroth-order optimization.

A prominent zeroth-order optimization algorithm is the Nesterov and Spokoiny's random search method~\cite{Polyak1987, NesterovGS, Larson2019, Nemirovsky83} developed based on the concept of Gaussian smoothing (GS), whose definition is recalled here for convenience.
\begin{definition}[{\hspace{1sp}\cite[section 2]{NesterovGS}}]\label{def_GSfunction}
Let $\sigma > 0$ and $f:\R^d\to \R$ be a Lebesgue measurable function. The Gaussian smoothing of $f$ is defined as
		\[
		f_\sigma(x)= \mathbb{E}_{u\sim\mathcal{N}(0,I)}[f(x + \sigma u)],
		\]
		where $\mathcal{N}(0,I)$ denotes the $d$-dimensional standard Gaussian distribution. 	
\end{definition}
As a convolution of $f$ with the Gaussian kernel, the GS $f_\sigma$ enjoys many desirable properties. For example, it was shown in \cite[section~2]{NesterovGS} to inherit convexity and Lipschitz continuity from $f$. Moreover, it was shown that $\nabla f_\sigma$ is Lipschitz continuous whenever $f$ is globally Lipschitz. This latter fact was leveraged in \cite[section~7]{NesterovGS} to establish the \emph{first worst-case complexity result} for a (stochastic) zeroth-order method for minimizing a nonsmooth nonconvex \emph{globally} Lipschitz function. The work \cite{NesterovGS} has stimulated a surge of studies on GS-based zeroth-order optimization algorithms; see, e.g, \cite{Lin2022, Maggiar2018, Bala2022, Berahas2022, Jongeneel21, Osher2022, Starnes2023a, Starnes2023b, DereventsovDGS_DFO_GS2022}. {\color{revise} It should however be pointed out that the advantage of GS-based algorithms over classical finite-difference methods is still under discussion~\cite{Scheinberg22}.}

To the best of our knowledge, most existing works on GS-based zeroth-order optimization algorithms, if not all, require the objective function itself, its gradient, or its Hessian to be Lipschitz continuous.
Such assumptions not only ensure that the GS is well-defined and its gradient can be unbiasedly approximated by random samples of $ f(x + \sigma u)u/\sigma$ or $ (f(x + \sigma u) - f(x))u/\sigma$ (with $u\sim \mathcal{N}(0,I)$), but also play a crucial role in the convergence analysis of the corresponding GS-based zeroth-order optimization algorithms. Nonetheless, these Lipschitz assumptions may not hold in many practical applications, {\color{revise}including hyperparameter tuning~\cite{Ehrhardt2020}, distributionally robust optimization~\cite{JinZhangWangWang21, taskesen2021sequential}, neural network training~\cite{ZhangHeSraJad19}, adversarial attacks~\cite{Chen2017} and $\mathcal{H}_\infty$ control~\cite{Guo2023}}. It is thus important to study less stringent Lipschitz assumptions to widen the applicability of zeroth-order optimization.

A similar issue concerning Lipschitz assumptions also arises in the context of first-order methods, where the global Lipschitzness of the gradient is instrumental to the algorithmic design and analysis.
As an attempt to relax the Lipschitz requirement in the study of first-order methods, various notions of generalized smoothness~\cite{JinZhangWangWang21, ChenZhouLiangLu23, ZhangHeSraJad19, MiKhWaDeGo24} have been recently proposed and led to the development and analysis of new first-order methods for these classes of generalized smooth functions.
While it may be {\color{rerevise}tempting} to adapt these notions to the study of zeroth-order optimization, it is unclear how this can be done even for the special case of GS-based zeroth-order optimization algorithms.

{\color{revise}Recently, a class of locally Lipschitz functions with Lipschitz modulus growing at most \emph{polynomially} was introduced in \cite[Assumption~1]{Bolte23} to study stochastic optimization. In this paper, we further study this class of locally Lipschitz functions,}
and develop new GS-based zeroth-order optimization algorithms for minimizing this class of functions. Our main contributions are threefold.
\begin{enumerate}
  \item We {\color{revise}study} the class of subdifferentially polynomially bounded (SPB) functions, which is the subclass of locally Lipschitz functions with a Lipschitz modulus that grows at most \emph{polynomially}. The class of SPB functions is rich, {\color{revise}encompassing not only all functions that are Lipschitz, or gradient-Lipschitz, or Hessian-Lipschitz, but also certain nonsmooth locally Lipschitz functions, such as functions arising from neural networks; see Examples~\ref{remark0612}\ref{remark0612_v}-\ref{remark0612_vi}.}
We show that if $f$ is SPB, then its GS $f_\sigma$ is well-defined and continuously differentiable; moreover, $f_\sigma$ and its partial derivatives are SPB too. We also establish a relationship between $\nabla f_\sigma$ and the Goldstein $\delta$-subdifferential of an SPB function $f$, which allows us to quantify the approximate stationarity of a point $x$ with respect to $f$ by measuring $\nabla f_\sigma (x)$. The Goldstein $\delta$-subdifferential is a commonly used subdifferential for studying stationarity of nonsmooth functions \cite{Goldstein1977,Zhang2020}, and our result can be viewed as an extension of \cite[Theorem~2]{NesterovGS} and \cite[Theorem~3.1]{Lin2022} from globally Lipschitz to SPB functions.

\item We {\color{revise}devise GS-based zeroth-order algorithms for SPB minimization under two different settings: the constrained convex setting (where the objective function $f$ is convex) and the unconstrained non-convex setting (where the feasible region $\Omega = \mathbb{R}^d$).}
Our {\color{revise}algorithms} update the iterate $x^k$ by moving along an approximate negative gradient direction with an adaptive stepsize depending inversely on a polynomial of $\|x^k\|$, and the approximate gradient is {\color{revise}obtained as} the random vector $ (f(x^k + \sigma u) - f(x^k))u/\sigma$ with $u\sim \mathcal{N} (0,I)$. We analyze the iteration complexity of the proposed algorithms.
The crux of our analysis is a novel descent lemma for $f_\sigma$ analogous to the standard descent lemma for Lipschitz differentiable functions, where the Lipschitz constant is replaced by a polynomial function.

\item {\color{revise} In the unconstrained non-convex setting, the above-mentioned complexity result is with respect to the GS $f_\sigma$ but not the original objective function $f$. Therefore, we also analyze the iteration complexity of our proposed algorithms for computing a $(\delta, \epsilon)$-stationary point, a notion of approximate Goldstein stationary point, with respect to the original objective function.}
\end{enumerate}

The remainder of this paper unfolds as follows. We present the notation and preliminary materials in section~\ref{sec2}. Section~\ref{sec3orig} introduces subdifferentially polynomially bounded functions and studies their properties in relation to GS and Goldstein $\delta$-stationarity.  In {\color{revise}sections~\ref{sec4} and \ref{sec44}}, we prove the descent lemma and develop our GS-based zeroth-order algorithms for {\color{revise}minimizing} SPB functions.

\section{Notation and preliminaries}\label{sec2}
Throughout this paper, we let $\mathbb{R}^{d}$ denote the Euclidean space of dimension $d$ equipped with the standard inner product $\langle\cdot,\cdot\rangle$. For any $x\in \R^d$, we let $\|x\|$ denote its Euclidean norm, and $\mathbb{B}(x,r)$ denote the closed ball in $\R^d$ with center $x$ and radius $r\geq 0$. We use $\B_r$ to denote $\mathbb{B}(0,r)$, and further use $\mathbb{B}$ to denote $\B_1$. We let $I=[e_1,e_2,\cdots,e_d]$ denote the $d\times d$ identity matrix, where $e_i\in \R^d$ is the $i$-th canonical basis vector for $i = 1,\dots,d$, i.e., $(e_i)_j = 1$ if $j=i$ and $(e_i)_j = 0$ otherwise.

For a subset $D\subseteq \R^d$, we let $D^c$, $\bd D$ and ${\rm conv}(D)$ denote its complement, boundary and convex hull, respectively; we also denote the characteristic function of $D$ by
\[
	\mathds{1}_{D}(x)=
	\begin{cases}
		1& {\rm if}\ x\in{D},\\
		0& {\rm if}\ x\notin{D}.
	\end{cases}
	\]
	
{\color{revise}For a closed set $S\subseteq \R^d$,
the distance from an $x\in \R^d$ to $S$ is defined as ${\rm dist}(x,S) = \inf_{y\in S}\|x - y\|$. For a closed convex set $S$, the (unique) projection of an $x\in \R^d$ onto $S$ is denoted by $P_S(x)$; also,} the normal cone of $S$ at any $x\in S$ is defined as
	\begin{equation*}
	N_{S}(x)= \{y\in \R^d:\; \langle y, u - x\rangle\le 0\ \ \forall u\in S\}.
	\end{equation*}

	
For a locally Lipschitz function $f: \mathbb{R}^{d}\rightarrow \mathbb{R}$, the Clarke directional derivative of $f$ (see \cite[Page~25]{Clarke1983}) at any $x\in \R^d$ in the direction $v\in \mathbb{R}^d$ is defined as
	\[
	f^{\rm o}(x;v)= \limsup_{x'\rightarrow x,t\downarrow 0} \frac{f(x'+tv)-f(x')}{t},
	\]
and the Clarke subdifferential of $f$ (see \cite[Page~27]{Clarke1983}) at $x$ is the set
\begin{equation*}
  \partial_Cf(x) = \{s\in \R^d:\; \langle s,v\rangle\le f^{\rm o}(x;v)\ \ \ \forall v\in \R^d\}.
\end{equation*}
The Clarke directional derivative and Clarke subdifferential are related as follows:
\begin{equation*}
f^{\rm o}(x;v)= \max_{s\in \partial_C f(x)} \langle s,v\rangle;
\end{equation*}
also, letting ${\color{revise}\Upsilon_f}$ be the set of points at which $f$ is not differentiable, we have
\begin{equation}\label{def_Clarkesubdiff}
	\partial_C f(x) = {\rm conv}\big(\big\{ s\in \mathbb{R}^d: \exists\{x^k\} \subset \R^d\setminus{\color{revise}\Upsilon_f} ~{\rm with}~x^k \rightarrow x~{\rm and}~\nabla f(x^k)\to s\big\}\big);
\end{equation}
see \cite[Propositions~2.1.2(b)]{Clarke1983} and \cite[Theorem~2.5.1]{Clarke1983}.

Next, for any $\delta > 0$, the Goldstein $\delta$-subdifferential \cite{Goldstein1977} of $f$ at $x\in \R^d$ is the set
\begin{equation}\label{def_Gsubdiff}
	\partial^\delta_G f(x) = {\rm conv}\bigg(\,{\bigcup}_{y \in \mathbb{B}(x, \delta)} \partial_C f(y)\bigg).
\end{equation}
Note that at any $x\in \R^d$, both the Clarke subdifferential and Goldstein $\delta$-subdifferential are compact convex sets.

\section{Subdifferentially polynomially bounded functions}\label{sec3orig}\
{\color{revise}In this section, we study} the class of subdifferentially polynomially bounded (SPB) functions.\footnote{{\color{revise}Condition \eqref{definitioneq01} with ${\rm R}_1 = {\rm R}_2$ is equivalent to \cite[Assumption 1.1]{Bolte23} when the function $f(x,s)$ there is constant in $s$, i.e., in the deterministic setting.}}
\begin{definition}[{{Subdifferentially polynomially bounded functions}}]\label{def:polybd}
Let $f: \R^d \rightarrow \R$ be locally Lipschitz continuous. We say that $f$ is subdifferentially polynomially bounded (SPB) if there exist ${\rm R_1}\geq0$, ${\rm R_2} > 0$ and an integer {\color{revise}$m\geq0$ with ${\rm R}_1 = 0$ if and only if $m = 0$} such that\footnote{{\color{revise} We adopt the convention $0^0=1$ when $x = 0$ and $m=0$ (in which case ${\rm R_1}=0$).}}
\begin{equation}\label{definitioneq01}
\sup\limits_{\zeta\in\partial_{C}f(x)}\|\zeta\|\leq {\rm R_1}\|x\|^{m}+{\rm R_2}\ \ \ \ \ \ \forall x\in \R^d.
\end{equation}
The class of SPB functions on $\R^d$ is denoted by $\spb$.
\end{definition}

Note that using calculus rules for Clarke subdifferential (see Corollary~2 of \cite[Proposition~2.3.3]{Clarke1983}), one can show that $\spb$ is a vector space. Also, $\spb$ generalizes the class of globally Lipschitz functions, which correspond to the case of ${\rm R_1} = 0$ in \eqref{definitioneq01}; see Example~\ref{remark0612}\ref{remark0612_i} below.
In fact, the SPB class is much richer than that and covers a wide variety of functions that arise naturally in many contemporary applications.
Here, we present some concrete examples of SPB functions.

\begin{example}\label{remark0612}
\begin{enumerate}[label={\rm (\roman*)}]
  \item\label{remark0612_i} If $f:\R^d\to \R$ is globally Lipschitz continuous with Lipschitz continuity modulus $L > 0$, then we have from \cite[Proposition~2.1.2(a)]{Clarke1983} that $\sup\limits_{u\in\partial_{C}f(x)}\|u\|\le L$ for all $x\in \R^d$. Consequently, $f$ is SPB.
  \item\label{remark0612_ii} Every polynomial function is SPB.
  \item\label{remark0612_iii} Any continuously differentiable function with a Lipschitz gradient is SPB. To see this, let $g$ be such a function, {\color{revise}then} there exists $L > 0$ such that
      \[
      \|\nabla g(x)\|\le \|\nabla g(x) - \nabla g(0)\| + \|\nabla g(0)\|\le L\|x\| + \|\nabla g(0)\|\ \ \ \ \forall x \in \R^d,
      \]
      showing that $g$ is SPB {\color{revise}(with $m = 1$). The converse is however not true, as an SPB function with $m=1$ is not necessarily differentiable.}
  \item\label{remark0612_iv} Let $f=g\circ h$, where $g:\mathbb{R}^{n}\rightarrow\mathbb{R}$ and $h:\mathbb{R}^{d}\rightarrow\mathbb{R}^{n}$. Assume that $g$ and all component functions of $h$ are SPB. Then one can deduce from \cite[Theorem~2.3.9]{Clarke1983} that $f$ is SPB.

\item \label{remark0612_v}
In machine learning, one is often interested in approximating the unknown relationship between an independent variable $x\in \mathbb{R}^d$ and a dependent variable $y\in\mathbb{R}$.
An $L$-layer neural network is a parametric approximation of the form
\begin{equation}\label{NNdef}
y = \psi(x; w)=\varrho_{L}(W_{L}(\varrho_{L-1}(W_{L-1}(\cdots\varrho_{1}(W_{1}(x))\cdots)))) ,
\end{equation}
where for $\ell = 1,\dots, L$, $\varrho_{\ell}:\mathbb{R}\to\mathbb{R}$ is the activation function for the $\ell$-th layer (for any vector $z$, the notation $\varrho_\ell (z)$ is understood as the vector obtained by applying the activation function $\varrho_\ell$ entrywise to $z$), $W_{\ell}:\mathbb{R}^{p_{\ell}}\rightarrow\mathbb{R}^{p_{\ell+1}}$ is an affine mapping for some positive integers $p_{\ell}$ and $p_{\ell+1}$ (with $p_1 = d$ and $p_{L+1} = 1$), and $w$ is called the parameter and represents the vector of all coefficients defining the maps $W_1,\dots, W_L$; see \cite[section~6.2]{Bolte2021} for details. Common activation functions include: $\varrho(t)=t$ (often used for the output layer), $\varrho(t)=\tanh(t)$, $\varrho(t)=\ln(1+e^{t})$, $\varrho(t)=\max\{0,t\}$, $\varrho(t)=\max\{0,t\}+\alpha\min\{0,t\}$ with $\alpha>0$, $\varrho(t)=\frac{1}{1+e^{-t}}$, and piecewise polynomial functions. With any of these activation functions, Example~\ref{remark0612}\ref{remark0612_iv} implies that the neural network function $\psi(\cdot; w)$ is SPB for any fixed parameter $w$.

\item\label{remark0612_vi} Suppose that we are given a sample $\{(x_i, y_i)\}_{i=1}^n$ of $n$ data points for approximating the unknown relationship between $x$ and $y$. Naturally, we want to find the best parameter $w$ so that the function $\psi(\cdot; w)$ in \eqref{NNdef} fits the sample data as well as possible, a process called ``training". A popular formulation for the best parameter $w$ is given by the least squares criterion:
\begin{equation*}\label{opt:NN}
\min_{w}\; \sum_{i=1}^n (y_i - \psi(x_i; w))^2.
\end{equation*}	
Again by Example~\ref{remark0612}\ref{remark0612_iv}, the objective function in this problem is SPB.
\end{enumerate}
\end{example}

\subsection{Gaussian smoothing of SPB functions}\label{sec:GS}

We next study the properties of SPB functions in relation to Gaussian smoothing (GS) \cite{Polyak1987, NesterovGS, Larson2019, Nemirovsky83}.
More precisely, we will show that for any SPB function, both the GS and its gradient are well-defined and that the class $\spb$ is closed under the GS transformation. Towards that end, we record a simple property of SPB functions that will be repeatedly used in the paper. Specifically, we express the Lipschitz modulus of an SPB function in terms of a sum of functions in $x$ and the \emph{displacement} $y - x$. This explicit dependence on the displacement is crucial for our subsequent analysis, especially in section~\ref{sec4}.
\begin{lemma}\label{prop:useful}
  Let $f\in \spb$ with parameters ${\rm R}_1$, ${\rm R}_2$ and $m$ as in \eqref{definitioneq01}. Then 
  \[
  |f(x)-f(y)|\leq(2^{m-1}{\rm R}_{1}\|x\|^{m}+2^{m-1}{\rm R}_{1}\|y - x\|^{m}+{\rm R}_{2})\|x - y\|\ \ \ \ \ \forall x,y\in \R^d.
  \]
\end{lemma}
\begin{proof}
  From \cite[Theorem~2.3.7]{Clarke1983}, we have
\begin{equation*}
f(x)-f(y)\in\{\langle \zeta,x - y\rangle:\; \zeta \in \partial_C f(x+\alpha (y-x)),\ \alpha \in (0,1)\}.
\end{equation*}
In view of this and \eqref{definitioneq01}, we deduce further that
\begin{align*}
&|f(x)-f(y)|\leq \sup_{\alpha\in (0,1)}\{{\rm R}_{1}\|x+\alpha (y - x)\|^{m}+{\rm R}_{2}\}\|x-y\|\\
&\leq(2^{m-1}{\rm R}_{1}\|x\|^{m}+2^{m-1}{\rm R}_{1}\|y - x\|^{m}+{\rm R}_{2})\|x - y\|,
\end{align*}
where the second inequality follows from {\color{revise}the fact that $\alpha\in (0,1)$ and} the convexity of the function $\|\cdot\|^m$ {\color{revise} when $m \ge 1$, and the inequality holds trivially when $m = 0$ (in which case $R_1 = 0$) with the convention $0^0 = 1$.}
\end{proof}
	

{\color{revise}When the function $f$ is not globally Lipschitz, its GS $f_\sigma$ is not necessarily defined. This is exemplified by the function $f(x) = e^{\|x\|^4}$.}
The theorem below asserts that for any SPB function $f$, the GS $f_\sigma$ and its gradient $\nabla f_\sigma$ are both well-defined. {\color{revise} An explicit formula for $\nabla f_\sigma$ is also proved. The crux for the proof of the formula lies in the interchangeability of differentiation and integration, which we achieve by using the theory of Schwartz {\color{rerevise}spaces} and tempered distributions. Roughly speaking, {\color{rerevise}a Schwartz space} consists of $C^\infty$ functions whose derivatives of any order decay faster than any polynomial (e.g., $e^{-\|\cdot\|^2/2}$; see \cite[Example 2.2.2]{Grafakos2008}), and tempered distributions form its topological dual.}

\begin{theorem}[{{Well-definedness of GS and its gradient}}]\label{lem053101}
Let $f\in \spb$ with parameters ${\rm R}_1$, ${\rm R}_2$ and $m$ as in \eqref{definitioneq01}. Then its GS $f_\sigma$, given in Definition~\ref{def_GSfunction}, is well-defined. Moreover, the gradient of $f_\sigma$ is given by
	\begin{equation}\label{GSgradient}
		\nabla f_\sigma(x) = {\color{revise}\sigma^{-1}}\mathbb{E}_{u\sim\mathcal{N}(0,I)}[f(x + \sigma u)u]
	\end{equation}
and is well-defined and continuous.
\end{theorem}
\begin{proof}For any $x\in \R^d$, we have
\begin{align}\label{1stbd}
&\mathbb{E}_{u\sim\mathcal{N}(0,I)}[|f(x+\sigma u)|]\leq\mathbb{E}_{u\sim\mathcal{N}(0,I)}[|f(x+\sigma u)-f(x)|]+|f(x)|\notag\\
&\leq\mathbb{E}_{u\sim\mathcal{N}(0,I)}[(2^{m-1}{\rm R}_{1}\|x\|^{m}+2^{m-1}{\rm R}_{1}\sigma^{m}\|u\|^{m}+{\rm R}_{2})\cdot \sigma\|u\|]+|f(x)| < \infty,
\end{align}
where the second inequality {\color{revise}follows from Lemma~\ref{prop:useful}.}
Therefore, $f_{\sigma}$ is well-defined.

We next prove \eqref{GSgradient} and the well-definedness of the expectation there. First, by the definition of GS, we have
\begin{align}\label{eq042701}
f_{\sigma}(x)=\frac{1}{(2\pi)^\frac{d}2}\int_{\R^d}f(x+\sigma u)e^{-\frac{\|u\|^{2}}{2}}du=\frac{1}{(2\pi)^\frac{d}2\sigma^{d}}\int_{\R^d}f(y)e^{-\frac{\|x-y\|^{2}}{2\sigma^{2}}}dy.
\end{align}
Note that {\color{revise}$e^{-\|\cdot\|^{2}/(2\sigma^{2})}$} is a Schwartz function on $\R^{d}$ according to \cite[Example 2.2.2]{Grafakos2008}. On the other hand, we have the following inequality for each $r > 0$ based on Lemma~\ref{prop:useful}:
\begin{align*}
 \int_{\mathbb{B}_{r}}|f(x)|dx&\le r^{d}\alpha(d)\cdot(|f(0)| + \max_{x\in\mathbb{B}_{r}}|f(x) - f(0)|)\\
 & \le r^{d}\alpha(d)\cdot(|f(0)| + [2^{m-1}{\rm R}_1r^m + {\rm R}_2]r) =: {\frak U}(r),
\end{align*}
where $\alpha(d)$ is the volume of the $d$-dimensional unit ball $\B$.
Since ${\frak U}(r) = O(r^{m+d+1})$ as $r\to \infty$, in view of \cite{Stein2011},\footnote{Specifically, see the last example on page 106.} we see that ${\frak F}(g):=\frac{1}{(2\pi)^\frac{d}2\sigma^{d}}\int_{\R^{d}}f(y)g(y)dy$ is a continuous linear functional on {\color{rerevise}Schwartz spaces} (i.e., a tempered distribution).

The next part follows {\color{revise}closely the proof} of \cite[Theorem~2.3.20]{Grafakos2008}, which is included for self-containedness. Specifically, from \eqref{eq042701}, we see that for any $h\in \R\backslash\{0\}$ and any $i \in \{1,\ldots,d\}$,
\begin{align*}
  \frac{f_{\sigma}(x+he_{i})-f_{\sigma}(x)}{h}&=\frac{1}{(2\pi)^\frac{d}2\sigma^{d}}\int_{\R^{d}}f(y)\frac{e^{-\frac{\|x+he_{i}-y\|^{2}}{2\sigma^{2}}}-e^{-\frac{\|x-y\|^{2}}{2\sigma^{2}}}}{h}dy\\
  &={\frak F}((e^{-\frac{\|x+he_{i}-y\|^{2}}{2\sigma^{2}}}-e^{-\frac{\|x-y\|^{2}}{2\sigma^{2}}})/h).
\end{align*}
Since $(e^{-\frac{\|x+he_{i}-y\|^{2}}{2\sigma^{2}}}-e^{-\frac{\|x-y\|^{2}}{2\sigma^{2}}})/h\rightarrow -\frac{x_{i}-y_{i}}{\sigma^2}e^{-\frac{\|x-y\|^{2}}{2\sigma^{2}}}$ as $h \to 0$ in {\color{rerevise}Schwartz spaces} according to \cite[Exercise~2.3.5(a)]{Grafakos2008} and ${\frak F}$ is a tempered distribution, we conclude upon passing to the limit as $h \to 0$ in the above display that
\begin{align*}
  \nabla f_{\sigma}(x)=\bigg[{\frak F}\bigg(-\frac{x_{i}-y_{i}}{\sigma^2}e^{-\frac{\|x-y\|^{2}}{2\sigma^{2}}}\bigg)\bigg]_{i=1}^d
  &= -\frac{1}{(2\pi)^\frac{d}2\sigma^{d+2}}\int_{\R^d}f(y)(x-y)e^{-\frac{\|x-y\|^{2}}{2\sigma^{2}}}dy\\
  &=\frac{1}{(2\pi)^\frac{d}2\sigma}\int_{\R^d}f(x+\sigma u)ue^{-\frac{\|u\|^{2}}{2}}du.
\end{align*}
This proves \eqref{GSgradient} and the well-definedness of the integral.

Finally, the continuity of $\nabla f_{\sigma}$ follows immediately from the above integral representation and the dominated convergence theorem, where the required integrability assumption can be established in a similar way to \eqref{1stbd}.
\end{proof}

The next result shows in particular that $\spb$ is closed under the GS transformation and that if $f$ is SPB, so are its partial derivatives. {\color{revise}The case $m = 0$ was already established in \cite{NesterovGS}.}
\begin{theorem}\label{prop:gsgradient}
  Let $f\in \spb$ with parameters ${\rm R}_1$, ${\rm R}_2$ and $m$ as in \eqref{definitioneq01} and let $f_\sigma$ be defined in Definition~\ref{def_GSfunction}. Then the following statements hold.
\begin{enumerate}[label={\rm (\roman*)}]
  \item\label{prop:gsgradient_i} It holds that
  \begin{equation}\label{eq110601}
|f_{\sigma}(x)- f_{\sigma}(y)|\leq (\mathfrak{A}+\mathfrak{B}\|x\|^{m}+\mathfrak{C}\|y-x\|^{m})\|x-y\| \ \ \ \ \ \forall x,y\in \R^d,
  \end{equation}
  where $ \mathfrak{A} = 2^{2m-2}{\rm R}_{1}\sigma^{m}(m+d)^{\frac{m}{2}} + {\rm R_2}$, $\mathfrak{B}=2^{2m-2}{\rm R}_{1}$ and $\mathfrak{C}=2^{m-1}{\rm R}_{1}$. In particular, $f_\sigma$ is SPB.

  \item\label{prop:gsgradient_ii} It holds that
  \begin{equation}\label{eq102801}
\|\nabla f_{\sigma}(x)-\nabla f_{\sigma}(y)\|\leq (\mathcal{A}+\mathcal{B}\|x\|^{m}+\mathcal{C}\|y-x\|^{m})\|x-y\|\ \ \ \ \ \forall x,y\in \R^d,
  \end{equation}
where $ \mathcal{A} = 2^{2m-2}{\rm R}_{1}\sigma^{m-1}(m+1+d)^{\frac{m+1}{2}}+{\color{revise}\sigma^{-1}}\mathrm{R}_{2}\sqrt{d}$, $\mathcal{B}=2^{2m-2}{\color{revise}\sigma^{-1}}{\rm R}_{1}\sqrt{d}$ and $\mathcal{C}=2^{m-1}{\color{revise}\sigma^{-1}}{\rm R}_{1}\sqrt{d}$. In particular, $\tfrac{\partial f_\sigma}{\partial x_i}$ is SPB for any $i$.
\end{enumerate}
\end{theorem}
\begin{proof}
{\color{revise}The case $m=0$ (in which case we have $\mathrm{R}_{1}=0$) was studied in \cite{NesterovGS}, with item (i) proved in the display before \cite[Eq.~(12)]{NesterovGS}, and item (ii) proved in \cite[Lemma~2]{NesterovGS}.}

{\color{revise}We next consider the case $m \ge 1$.} We first observe from Lemma \ref{prop:useful} that for every $x$, $y$ and $u\in \R^d$,
\begin{align}\label{eq110602}
&|f(x+\sigma u)-f(y+\sigma u)|\notag\\
&\leq (2^{m-1}\mathrm{R}_{1}\|x+\sigma u\|^{m}+2^{m-1}\mathrm{R}_{1}\|y-x\|^{m}+\mathrm{R}_{2})\|y-x\|\notag\\
& \le(2^{2m-2}{\rm R}_{1}\sigma^{m}\|u\|^{m}+2^{2m-2}{\rm R}_{1}\|x\|^{m}+2^{m-1}{\rm R}_{1}\|y-x\|^{m}+{\rm R}_{2})\|x-y\|,
\end{align}
where the second inequality follows from the convexity of $\|\cdot\|^m$ when $m\ge1$.

To prove~\ref{prop:gsgradient_i}, from \eqref{eq110602}, one has for any $x\neq y$ that
\begin{align*}
&\frac{|f_{\sigma}(x)-f_{\sigma}(y)|}{\|x - y\|}\le\frac{\mathbb{E}_{u\sim\mathcal{N}(0,I)}[|f(x+\sigma u)-f(y+\sigma u)|]}{\|x - y\|}\\
&\le\mathbb{E}_{u\sim\mathcal{N}(0,I)}[(2^{2m-2}{\rm R}_{1}\sigma^{m}\|u\|^{m}+2^{2m-2}{\rm R}_{1}\|x\|^{m}+2^{m-1}{\rm R}_{1}\|y-x\|^{m}+{\rm R}_{2})]\\
&=(2^{2m-2}\mathrm{R}_{1}\|x\|^{m}+2^{m-1}\mathrm{R}_{1}\|y-x\|^{m}+\mathrm{R}_{2})
+2^{2m-2}\mathrm{R}_{1}\sigma^{m}\mathbb{E}_{u\sim\mathcal{N}(0,I)}[\|u\|^{m}]\\
&\le(2^{2m-2}\mathrm{R}_{1}\|x\|^{m}+2^{m-1}\mathrm{R}_{1}\|y-x\|^{m}+\mathrm{R}_{2})
+2^{2m-2}\mathrm{R}_{1}\sigma^{m}(m+d)^{\frac{m}{2}},
\end{align*}
where {\color{revise}the last inequality} follows from \cite[Lemma~1]{NesterovGS}. This proves \eqref{eq110601}.

Now, fix any $x\in \R^d$ and $v\in \R^d$ with $\|v\|=1$. We have for any $\xi \in \partial_C f_\sigma(x)$ that
\begin{align*}
  \langle \xi,v\rangle & \le \limsup_{x'\to x, t\downarrow 0}\frac{f_\sigma(x' + tv)\! - \!f_\sigma(x')}{t}\le \limsup_{x'\to x, t\downarrow 0}(\mathfrak{A}\!+\!\mathfrak{B}\|x'\|^{m}\!+\!\mathfrak{C}t^m) \le \mathfrak{A}+\mathfrak{B}\|x\|^{m}.
\end{align*}
Consequently, it holds that $\|\xi\|\le \mathfrak{A}+\mathfrak{B}\|x\|^{m}$, showing that $f_\sigma \in \spb$.

To prove \ref{prop:gsgradient_ii}, we notice from \eqref{GSgradient} that
\begin{equation*}
\begin{aligned}
\|\nabla f_{\sigma}(x)-\nabla f_{\sigma}(y)\|\leq{\color{revise}\sigma^{-1}}\mathbb{E}_{u\sim\mathcal{N}(0,I)}[|f(x+\sigma u)-f(y+\sigma u)|\cdot \|u\|].
\end{aligned}
\end{equation*}
Combining the above display with \eqref{eq110602}, we have for any $x\neq y$ that
\begin{align*}
&\frac{\|\nabla f_{\sigma}(x)-\nabla f_{\sigma}(y)\|}{\|x - y\|}\\
& \le 2^{2m-2}{\rm R}_{1}\sigma^{m-1}\mathbb{E}_{u\sim\mathcal{N}(0,I)}[\|u\|^{m+1}]\\
& ~~~+{\color{revise}\sigma^{-1}}(2^{2m-2}{\rm R}_{1}\|x\|^{m}+2^{m-1}{\rm R}_{1}\|y-x\|^{m}+{\rm R}_{2})\mathbb{E}_{u\sim\mathcal{N}(0,I)}[\|u\|]\\
& \le 2^{2m-2}{\rm R}_{1}\sigma^{m-1}(m\!+\!1\!+\!d)^{\frac{m+1}{2}}\!+\!{\color{revise}\sigma^{-1}}(2^{2m-2}{\rm R}_{1}\|x\|^{m}\!+\!2^{m-1}{\rm R}_{1}\|y\!-\!x\|^{m}+{\rm R}_{2})\sqrt{d},
\end{align*}
where {\color{revise}the last inequality} follows from \cite[Lemma~1]{NesterovGS}. This proves \eqref{eq102801}.

The claim that $\tfrac{\partial f_\sigma}{\partial x_i}\in\spb$ can now be proved in a similar way to the proof of $f_\sigma\in \spb$ in item~\ref{prop:gsgradient_i}.
\end{proof}

\subsection{Approximate Goldstein stationarity}\label{sec3}

In this subsection, we explore the relationship between the GS gradient 
and the Goldstein $\delta$-subdifferential for SPB functions. We start with the following auxiliary lemma concerning the tail of the Gaussian integral. We let $W_{-1}$ denote the negative real branch of the Lambert $W$ function (see, e.g., \cite{Wfunction2,wfunc_old1955,wfunc_old1973}); this function is defined as the inverse of the function $t\mapsto te^t$ with domain $[-1/e,0)$ and range $(-\infty,-1]$.

\begin{lemma}\label{lem052902}
For any $\nu>0$ and $M\geq \big[{\color{revise}-d\cdot W_{-1}}\big(- \nu^{\frac{2}{d}}/ (2\pi e) \big)\big]^{\frac{1}{2}}$, it holds that
\[
\int_{\|u\|\ge M}e^{-\frac{\|u\|^{2}}{2}}du\leq\nu;
\]
here, we set by convention that $W_{-1}(t) = 0$ if $t<-1/e$.
\end{lemma}
\begin{proof}
Fix any $\nu>0$ and $M\geq \wp:= \big[{\color{revise}-d\cdot W_{-1}}\big(- \nu^{\frac{2}{d}}/ (2\pi e)\big)\big]^{\frac{1}{2}}$.
For any $R \ge 0$,
\begin{align*}
  \int_{\|u\|\geq R}e^{-\frac{\|u\|^{2}}{2}}du=(2\pi)^\frac{d}2\cdot\frac{1}{(2\pi)^\frac{d}2}\int_{\|u\|\geq R}e^{-\frac{\|u\|^{2}}{2}}du=(2\pi)^\frac{d}2[1-F(R^{2};d)],
\end{align*}
where $F(\cdot;k)$ is the cumulative distribution function of the chi-squared distribution with $k$ degrees of freedom. Thus, the desired conclusion is equivalent to
\begin{equation}\label{equiv_conclu}
1-\nu(2\pi)^{-\frac{d}2}\le F(M^{2};d).
\end{equation}
We now prove~\eqref{equiv_conclu}. Note that if $\nu \ge (2\pi)^\frac{d}2$, then we have $1-\nu(2\pi)^{-\frac{d}2}\le 0 \le F(M^{2};d)$. Hence, \eqref{equiv_conclu} is valid in this case. We next consider the case $\nu<(2\pi)^\frac{d}2$. In this case, we have $\nu^\frac2{d} < 2\pi$ and hence $M\geq\wp > \sqrt{d}$.\footnote{The second inequality holds because $W_{-1}(t)< -1$ when $t\in (-1/e,0)$; see \cite[Page~2]{Loczi}.} Then, from \cite[Lemma 2.2]{Dasgupta2003} (see also \cite[Proposition 5.3.1]{So2007}), we have
\begin{align}
1-F(M^{2};d) = 1-F\bigg(\frac{M^{2}}{d}d;d\bigg)\le\bigg(\frac{M^{2}}{d}e^{1-\frac{M^2}{d}}\bigg)^{\frac{d}{2}}.\label{eq042802}
\end{align}
Now, note that we have the following equivalence for any $R\ge 0$
\begin{align}
  \bigg(\frac{R^2}{d}e^{1-\frac{R^2}{d}}\bigg)^{\frac{d}{2}}\le\frac{\nu}{(2\pi)^\frac{d}2}
  \Longleftrightarrow\bigg(-\frac{R^2}{d}\bigg)e^{-\frac{R^2}{d}}\geq-e^{-1}\frac{\nu^{\frac2d}}{2\pi}.\label{eq042901}
\end{align}
Then we have from the definition of Lambert $W$ function that the rightmost inequality (and hence both inequalities) in \eqref{eq042901} holds with $R = M$ since $M\geq\wp$. Thus, the desired conclusion follows from~\eqref{eq042901} and \eqref{eq042802}.
\end{proof}

In the next theorem, we show that for all sufficiently small $\sigma > 0$, {\color{revise}some} Goldstein $\delta$-subgradients can be approximated by the GS gradient $\nabla f_\sigma$. Specifically, we derive a sufficient condition on $\sigma$, in the form of an \emph{explicit} upper bound depending on $\delta$ and $\varepsilon$, for the GS gradient to reside in {\color{revise}a $(1+\|x\|^m)\varepsilon$} neighborhood of the Goldstein $\delta$-subdifferential.
{\color{revise}
This inclusion is proved as follows. We first express the GS gradient as an expectation of the original gradient $\nabla f (x + \sigma u)$ (which exists almost everywhere by {\color{rerevise}Rademacher's theorem}) with respect to the random vector $u \sim\mathcal{N}(0,  I)$, see~\eqref{importantintegral} below. It follows from Definition~\ref{def_Gsubdiff} that $\nabla f (x + \sigma u)$ constitutes a Goldstein subgradient for any realization of $u$ close enough to $x$. Therefore, by dividing the expectation in~\eqref{importantintegral} into two integrals, one over a small ball centered at $x$ and the other over the complement, we can see that the GS gradient is an approximate Goldstein subgradient. The integral over the complement contributes to the approximation error, which can be controlled by using Lemma~\ref{lem052902}.}

Similar results have been derived under a globally Lipschitz continuity assumption on $f$; see, e.g., \cite[Theorem~2]{NesterovGS} and \cite[Theorem~3.1]{Lin2022}.
In particular, the proof of \cite[Theorem~3.1]{Lin2022} was based on an analogue of \eqref{importantintegral} for globally Lipschitz continuous $f$.
We would also like to point out that the representation~\eqref{importantintegral} can be seen as a variant of general results on convolution and differentiation such as \cite[section 4.2.5]{Vladimirov2002} and \cite[Lemma 9.1]{Brezis2011}.
Here we include an elementary proof {\color{revise}of \eqref{importantintegral}} to highlight the role of polynomial boundedness of the subdifferential.

\begin{theorem}[{{GS gradient as approximate Goldstein $\delta$-subgradient}}]\label{THBD01}
 Let $f\in \spb$ with parameters ${\rm R}_1$, ${\rm R}_2$ and $m$ as in \eqref{definitioneq01}. Let  $\nabla f_{\sigma}$ and $\partial_{G}^{\delta}f$ be given in \eqref{GSgradient} and \eqref{def_Gsubdiff}, respectively. Then the following hold.
\begin{enumerate}[label={\rm (\roman*)}]
  \item\label{THBD01:i} For every {\color{revise}$x\in \R^{d}$}, it holds that
  \begin{equation}\label{importantintegral}
  \nabla f_{\sigma}(x) = \mathbb{E}_{u\sim\mathcal{N}(0,I)}\left[\nabla f(x+\sigma u)\cdot\mathds{1}_{\mathfrak{D}_{\sigma}}(u)\right],
  \end{equation}
  where $\mathfrak{D}_\sigma= \{u\in \R^d:\; f \mbox{ is differentiable at } x + \sigma u\}$.

  \item\label{THBD01:ii} For every $\delta>0$ and $\varepsilon>0$, it holds that
\[
\nabla f_{\sigma}(x)\in\partial_{G}^{\delta}f(x)+{\color{revise}(1+\|x\|^{m})}\varepsilon\cdot\mathbb{B}\ \ \ \ \ \ \forall\, \sigma\in(0,\bar{\sigma}]  {\color{revise}\ \ {and}\ \ \forall\, x\in \R^d,}
\]
where
\begin{align}
\bar{\sigma}&=\min\bigg\{\bigg[\frac{\varepsilon}{2^{m+1}{\rm R}_{1}(m+d)^{\frac{m}{2}}}\bigg]^{\frac{1}{m}}, 1, \frac{\delta}{H}\bigg\},\label{defsigma}\\
  H&=\big[{\color{revise}-d\cdot W_{-1}(-\eta_{1}^{\frac{2}{d}}/(2\pi e))}\big]^{\frac{1}{2}},\label{defH} \\
  \eta_{1}&=\min\{\varepsilon\,{\cal P}^{-1}, (2\pi)^\frac{d}2-{\color{revise}0.5}\},
   \label{eta1}\\
   {\cal P} & = {\color{revise}4{\rm R_2}}+2^{m+1}{\rm R_1}(m+d)^{\frac{m}{2}} ,\label{PPP}
\end{align}
and $W_{-1}$ is the negative real branch of the Lambert $W$ function.\footnote{Since $\eta_1\le (2\pi)^\frac{d}2-\frac12$, we have $\eta_1^{2/d}/(2\pi e)< 1/e$ and hence $W_{-1}(-\eta_1^{2/d}/(2\pi e)) < -1$. Thus, $H\in (\sqrt{d},\infty)$.\label{footnote}}
\end{enumerate}
\end{theorem}
\begin{proof}
Fix any {\color{revise}$x\in \R^d$}.
From Theorem~\ref{lem053101}, the GS $f_\sigma$ and it gradient $\nabla f_\sigma$ are well-defined at $x$. For notational simplicity, for any $\sigma > 0$, let $\mathfrak{D}_\sigma=\{u\in\mathbb{R}^{d}:f \text{ is differentiable at } x+\sigma u \}$.
Then it follows from {\color{rerevise}Rademacher's theorem} that the complement $ \mathfrak{D}_\sigma^c$ has Lebesgue measure zero.

To prove \ref{THBD01:i}, note from Lemma~\ref{prop:useful} that for each $u$ and $h\in \mathbb{R}^d$
\begin{align}\label{PM051701TH3.1}
&|f(x+h+\sigma u)-f(x+\sigma u)|\leq [2^{m-1}{\rm R}_{1}\|x+\sigma u\|^{m} + 2^{m-1}{\rm R}_{1}\|h\|^m+{\rm R}_{2}]\|h\|\notag\\
&\leq(2^{2m-2}{\rm R}_{1}\|x\|^{m}+2^{2m-2}{\rm R}_{1}\sigma^{m}\|u\|^{m}+2^{m-1}{\rm R}_{1}\|h\|^{m}+{\rm R}_{2})\|h\|,
\end{align}
where the second inequality follows from the convexity of $\|\cdot\|^m$ when $m\ge1$ {\color{revise}and inequality holds when $m=0$  (in which case $R_1 = 0$) with the convention $0^{0}=1$}.

We prove \eqref{importantintegral} by contradiction. First, $\mathbb{E}_{u\sim\mathcal{N}(0,I)}[\nabla f(x+\sigma u)\cdot\mathds{1}_{\mathfrak{D}_{\sigma}}(u)]$ exists as $f$ is SPB. Suppose to the contrary $\mathbb{E}_{u\sim\mathcal{N}(0,I)}[\nabla f(x+\sigma u)\cdot\mathds{1}_{\mathfrak{D}_{\sigma}}(u)]\neq \nabla f_{\sigma}(x)$.
Define
\begin{equation*}
h_x= \mathbb{E}_{u\sim\mathcal{N}(0,I)}[\nabla f(x+\sigma u)\cdot\mathds{1}_{\mathfrak{D}_{\sigma}}(u)] - \nabla f_{\sigma}(x).
\end{equation*}
Then $h_x\neq 0$ and we have from the differentiability of $f_\sigma$ (see Theorem~\ref{lem053101}) that
\begin{align*}
0 &=\lim\limits_{t\rightarrow0}\frac{f_{\sigma}(x+th_x)-f_{\sigma}(x)-\langle\nabla f_{\sigma}(x),th_x\rangle}{\|th_x\|}\\
&=\frac{1}{(2\pi)^\frac{d}2}\lim\limits_{t\rightarrow0}\int_{\mathbb{R}^{d}}\frac{f(x+th_x+\sigma u)-f(x+\sigma u)-\langle\nabla f_{\sigma}(x),th_x\rangle}{\|th_x\|}\cdot e^{-\frac{\|u\|^{2}}{2}}du\\
&=\frac{1}{(2\pi)^\frac{d}2}\lim\limits_{t\rightarrow0}\int_{\mathfrak{D}_\sigma}\frac{f(x+th_x+\sigma u)-f(x+\sigma u)-\langle\nabla f_{\sigma}(x),th_x\rangle}{\|th_x\|}\cdot e^{-\frac{\|u\|^{2}}{2}}du\\
&\overset{\rm (a)}=\mathbb{E}_{u\sim\mathcal{N}(0,I)}\bigg[\frac{\langle\nabla f(x+\sigma u)- \nabla f_{\sigma}(x),h_x\rangle}{\|h_x\|}\cdot\mathds{1}_{\mathfrak{D}_{\sigma}}(u)\bigg]
\overset{\rm (b)}= \|h_x\|,
\end{align*}
where (a) follows from \eqref{PM051701TH3.1}, the dominated convergence theorem, and the fact that
\begin{align*}
&\lim_{t\to0}\bigg|\frac{f(x+th_x+\sigma u)-f(x+\sigma u)-\langle\nabla f_{\sigma}(x),th_x\rangle}{\|th_x\|} - \frac{\langle\nabla f(x+\sigma u)- \nabla f_{\sigma}(x),h_x\rangle}{\|h_x\|}\bigg|\\
&=\lim_{t\to 0}\bigg|\frac{f(x+th_x+\sigma u)-f(x+\sigma u)-\langle\nabla f(x+\sigma u),th_x\rangle}{\|th_x\|}\bigg| = 0,
\end{align*}
which holds thanks to the differentiability of $f$ at $x+\sigma u$ when $u \in \mathfrak{D}_\sigma$, and (b) follows from the definition of $h_x$. This contradicts the fact that $h_x \neq 0$. Thus, \eqref{importantintegral} holds.

We now prove \ref{THBD01:ii} by using the integral representation in \eqref{importantintegral} to relate $\nabla f_\sigma(x)$ to $\partial^\delta_Gf(x)$. To this end, we let $M>0$ and notice that for any $\sigma > 0$ we have
\begin{align}\label{HB051702}
\Delta_{h} &:= \mathbb{E}_{u\sim\mathcal{N}(0,I)}[\nabla f(x+\sigma u)\cdot \mathds{1}_{\mathfrak{D}_{\sigma}}(u)]\!-\!\frac{\mathbb{E}_{u\sim\mathcal{N}(0,I)}[\nabla f(x+\sigma u)\cdot \mathds{1}_{\mathfrak{D}_{\sigma}\cap \mathbb{B}_{M}}(u)]}{\mathbb{E}_{u\sim\mathcal{N}(0,I)}[\mathds{1}_{\mathfrak{D}_{\sigma}\cap \mathbb{B}_{M}}(u)]}\notag\\
&=\mathbb{E}_{u\sim\mathcal{N}(0,I)}[\nabla f(x+\sigma u)\cdot \mathds{1}_{\mathfrak{D}_{\sigma}\cap
\mathbb{B}_{M}^{c}}(u)]\notag\\
 & ~~+ \bigg(1-\frac{1}{\mathbb{E}_{u\sim\mathcal{N}(0,I)}[\mathds{1}_{\mathfrak{D}_{\sigma}\cap \mathbb{B}_{M}}(u)]}\bigg)\mathbb{E}_{u\sim\mathcal{N}(0,I)}[\nabla f(x+\sigma u)\cdot \mathds{1}_{\mathfrak{D}_{\sigma}\cap\mathbb{B}_{M}}(u)],
 \end{align}
where we recall that $\mathbb{B}_{M}=\{u\in\mathbb{R}^{d}:\|u\|\leq M\}$ and $\mathbb{B}_M^c$ is its complement.

For the first term on the second line of \eqref{HB051702}, we have
\begin{align}\label{ineq1}
&\left\|\mathbb{E}_{u\sim\mathcal{N}(0,I)}[\nabla f(x+\sigma u)\cdot \mathds{1}_{\mathfrak{D}_{\sigma}\cap\mathbb{B}_{M}^{c}}(u)]\right\|\notag\\
&\overset{\rm (a)}\le\mathbb{E}_{u\sim\mathcal{N}(0,I)}[({\rm R_{1}}\|x+\sigma u\|^{m}+{\rm R_{2}})\cdot \mathds{1}_{\mathfrak{D}_{\sigma}\cap\mathbb{B}_{M}^{c}}(u)]
\overset{\rm (b)}\leq \Xi_1 + \Xi_2
\end{align}
with
\begin{align*}
\Xi_1 &= (2^{m-1}{\rm R}_{1}{\color{revise}\|x\|^{m}}+{\rm R}_{2})\mathbb{E}_{u\sim\mathcal{N}(0,I)}[\mathds{1}_{\mathbb{B}_{M}^{c}}(u)],\\
\Xi_2 &= (2^{m-1}{\rm R}_{1}\sigma^{m})\mathbb{E}_{u\sim\mathcal{N}(0,I)}[\|u\|^{m}\cdot\mathds{1}_{\mathbb{B}_{M}^{c}}(u)],
\end{align*}
where we invoked \eqref{definitioneq01} in (a), and used the convexity of $\|\cdot\|^m$ when $m\ge1$ {\color{revise}and the fact that $\mathrm{R}_{1}=0$ when $m=0$}.
Now, observe that
\begin{equation*}
\mathbb{E}_{u\sim\mathcal{N}(0,I)}[\|u\|^{m}\cdot\mathds{1}_{\mathbb{B}_{M}^{c}}(u)]\leq\mathbb{E}_{u\sim\mathcal{N}(0,I)}[\|u\|^{m}]\leq(m+d)^{\frac{m}{2}},
\end{equation*}
where the second inequality follows from \cite[Lemma 1]{NesterovGS}.
Thus, {\color{revise}for $m\ge1$,} if we choose a finite positive $\sigma$ such that $\sigma\le \left[\frac{\varepsilon}{2^{m+1}{\rm R}_{1}(m+d)^{\frac{m}{2}}}\right]^{\frac{1}{m}}$, then
\begin{equation}\label{Xi2}
\begin{aligned}
\Xi_2 = (2^{m-1}{\rm R}_{1}\sigma^{m})\mathbb{E}_{u\sim\mathcal{N}(0,I)}[\|u\|^{m}\cdot\mathds{1}_{\mathbb{B}_{M}^{c}}(u)]
\leq{\color{revise}0.25\varepsilon\le 0.25\varepsilon(1+\|x\|^{m})}.
\end{aligned}
\end{equation}
{\color{revise}As for $m=0$, since $\mathrm{R}_{1}=0$, we conclude that for any $\sigma>0$, $\Xi_2=0\le0.25\varepsilon(1+\|x\|^{m})$. Thus, for $m\ge0$, if $\sigma \in (0,\infty)$ and $\sigma\le \left[\frac{\varepsilon}{2^{m+1}{\rm R}_{1}(m+d)^{\frac{m}{2}}}\right]^{\frac{1}{m}}$, then \eqref{Xi2} holds. }

Next, choose $\eta=\varepsilon(2\pi)^\frac{d}2\left({\color{revise}4 \max\{2^{m-1}\mathrm{R}_{1},\mathrm{R}_{2}\}}\right)^{-1}$.
Then, by setting $\nu=\eta$ in Lemma \ref{lem052902}, we see that whenever $M\geq \big[{\color{revise}-d\cdot W_{-1}}\big(-\eta^{\frac{2}{d}}/(2\pi e)\big)\big]^{\frac{1}{2}}$,
\begin{align}
\Xi_1 &= (2^{m-1}{\rm R}_{1}{\color{revise}\|x\|^{m}}+{\rm R}_{2})\mathbb{E}_{u\sim\mathcal{N}(0,I)}[\mathds{1}_{\mathbb{B}_{M}^{c}}(u)]\nonumber\\
&\leq {\color{revise} \max\{2^{m-1}\mathrm{R}_{1},\mathrm{R}_{2}\}(1+\|x\|^{m})\mathbb{E}_{u\sim\mathcal{N}(0,I)}[\mathds{1}_{\mathbb{B}_{M}^{c}}(u)]}
\le{\color{revise}0.25\varepsilon(1+\|x\|^{m})}.\label{Xi1}
\end{align}

Additionally, for the above $\eta$ and the $\eta_1$ in \eqref{eta1}, we have
      \[
       \eta_1 \le \varepsilon\left[4{\rm R_2}+2^{m+1}{\rm R_1}(m+d)^{\frac{m}{2}}\right]^{-1}\leq\eta.
      \]
Combining this conclusion with \eqref{ineq1}, \eqref{Xi2} and \eqref{Xi1}, we know that for any finite positive $\sigma\le\left[\frac{\varepsilon}{2^{m+1}{\rm R}_{1}(m+d)^{\frac{m}{2}}}\right]^{\frac{1}{m}}$ and $M\geq \big[{\color{revise}-d\cdot W_{-1}}\big(-\eta_{1}^{\frac{2}{d}}/(2\pi e)\big)\big]^{\frac{1}{2}}$ with $\eta_{1}$ as in \eqref{eta1}, we have
\begin{equation}\label{051707}
\|\mathbb{E}_{u\sim\mathcal{N}(0,I)}[\nabla f(x+\sigma u)\cdot\mathds{1}_{\mathfrak{D}_{\sigma}\cap\mathbb{B}^{c}_{M}}(u)]\|\leq {\color{revise}0.5\varepsilon(1+\|x\|^{m})}.
\end{equation}

We next estimate the second term on the second line of \eqref{HB051702}. We first notice that
$\int_{\mathbb{B}_{M}}{\color{revise}e^{-\|u\|^{2}/2}}du=(2\pi)^\frac{d}2-\int_{\B_M^c}{\color{revise}e^{-\|u\|^{2}/2}}du$. So, if $\int_{\B_M^c}{\color{revise}e^{-\|u\|^{2}/2}}du\leq (2\pi)^\frac{d}2-{\color{revise}0.5}$, then $\int_{\mathbb{B}_{M}}{\color{revise}e^{-\|u\|^{2}/2}}du\geq {\color{revise}0.5}$.
In view of Lemma~\ref{lem052902} and the definition of $\eta_1$, this happens when we choose
$M\geq \big[{\color{revise}-d\cdot W_{-1}}\big(-\eta_{1}^{\frac{2}{d}}/(2\pi e)\big)\big]^{\frac{1}{2}}$, since $\eta_1\le (2\pi)^\frac{d}2-{\color{revise}0.5}$.

On top of this choice of $M$, if we further choose $\sigma \le 1$, then the term in the last line of \eqref{HB051702} can be upper bounded as follows:
\begin{align}
&\bigg\|\bigg(1-\frac{1}{\mathbb{E}_{u\sim\mathcal{N}(0,I)}[\mathds{1}_{\mathbb{B}_{M}}(u)]}\bigg)\mathbb{E}_{u\sim\mathcal{N}(0,I)}[\nabla f(x+\sigma u)\cdot\mathds{1}_{\mathfrak{D}_{\sigma}\cap\mathbb{B}_{M}}(u)]\bigg\|\notag\\
&\leq\frac{\widetilde K}{\mathbb{E}_{u\sim\mathcal{N}(0,I)}[\mathds{1}_{\mathbb{B}_{M}}(u)] }\mathbb{E}_{u\sim\mathcal{N}(0,I)}[\|\nabla f(x+\sigma u)\|\cdot\mathds{1}_{\mathfrak{D}_{\sigma}\cap\mathbb{B}_{M}}(u)]\notag\\
&\overset{\rm (a)}\leq2(2\pi)^\frac{d}2 \widetilde K\mathbb{E}_{u\sim\mathcal{N}(0,I)}[\|\nabla f(x+\sigma u)\|\cdot\mathds{1}_{\mathfrak{D}_{\sigma}\cap\mathbb{B}_{M}}(u)]\notag\\
&\overset{\rm (b)}\leq2(2\pi)^\frac{d}2 \widetilde K\mathbb{E}_{u\sim\mathcal{N}(0,I)}[{\rm R_{1}}\|x+\sigma u\|^{m}+{\rm R_{2}}]\notag\\
&\overset{\rm (c)}\leq2(2\pi)^\frac{d}2 \widetilde K
\left[(2^{m-1}{\rm R}_{1}{\color{revise}\|x\|^{m}}+{\rm R}_{2})+2^{m-1}{\rm R}_{1}\sigma^{m}\mathbb{E}_{u\sim\mathcal{N}(0,I)}[\|u\|^{m}]\right]\notag\\
&\overset{\rm (d)}\leq2(2\pi)^\frac{d}2\widetilde K[(2^{m-1}{\rm R}_{1}{\color{revise}\|x\|^{m}}+{\rm R}_{2})+2^{m-1}{\rm R}_{1}(m+d)^{\frac{m}{2}}]\notag\\
&{\color{revise}\ \le 2(2\pi)^\frac{d}2\widetilde K[{\rm R}_{2}+2^{m-1}{\rm R}_{1}(m+d)^{\frac{m}{2}}](1+\|x\|^{m})},\label{upperbound}
\end{align}
where $\widetilde K := \mathbb{E}_{u\sim\mathcal{N}(0,I)}[\mathds{1}_{\mathbb{B}_{M}^{c}}(u)]$, and (a) holds because $(2\pi)^\frac{d}2\mathbb{E}_{u\sim\mathcal{N}(0,I)}[\mathds{1}_{\mathbb{B}_{M}}(u)] = {\color{revise}\int_{\B_M}e^{-\|u^2\|/2}du\ge 0.5}$, (b) holds upon using \eqref{definitioneq01} and enlarging the domain of integration, (c) follows from the convexity of $\|\cdot\|^m$ when $m\ge1$ {\color{revise}and the inequality holds trivially as an equality when $m=0$ because ${\rm R}_1= 0$}, (d) follows from \cite[Lemma 1]{NesterovGS} and the choice that $\sigma \le 1$. In view of \eqref{upperbound}, we can now invoke Lemma~\ref{lem052902}
to deduce that if we choose $M\geq \big[{\color{revise}-d\cdot W_{-1}}\big(-\eta_{1}^{\frac{2}{d}}/(2\pi e)\big)\big]^{\frac{1}{2}}$ with $\eta_1$ as in \eqref{eta1} and $\sigma \le 1$, then
\begin{align}
\bigg\|\bigg(\!1-\frac{1}{\mathbb{E}_{u\sim\mathcal{N}(0,I)}[\mathds{1}_{\mathbb{B}_{M}}(u)]}\!\bigg)\mathbb{E}_{u\sim\mathcal{N}(0,I)}\!\big[\nabla f(x\!\!+\!\!\sigma u)\cdot\mathds{1}_{\mathfrak{D}_{\sigma}\cap\mathbb{B}_{M}}(u)\big]\bigg\|\!\leq\!\frac{\varepsilon}{2}{\color{revise}(1\!\!+\!\!\|x\|^{m})}.\label{051708}
\end{align}
Thus, we conclude that \eqref{051707} and \eqref{051708} will both hold as long as we choose $M \ge H$ defined as in \eqref{defH} and $\sigma \le \widetilde\sigma = \min\big\{\big[\frac{\varepsilon}{2^{m+1}{\rm R}_{1}(m+d)^{\frac{m}{2}}}\big]^{\frac{1}{m}}, 1\big\}$. Hence, we have
\begin{equation}\label{hahaha}
\|\Delta_{h}\|\le {\color{revise}\varepsilon(1+\|x\|^{m})} \ \ {\rm whenever}\ \ M \ge H\ {\rm and}\ \sigma \le \widetilde\sigma,
\end{equation}
where $\Delta_{h}$ is given in \eqref{HB051702}.

Finally, for $M =H$ and any $\sigma\le\bar{\sigma}=\min\left\{\left[\frac{\varepsilon}{2^{m+1}{\rm R}_{1}(m+d)^{\frac{m}{2}}}\right]^{\frac{1}{m}}, 1, \frac{\delta}{H}\right\}$, we have $\sigma M \le \delta$. Hence, by \eqref{def_Clarkesubdiff} and the definition of Goldstein $\delta$-subdifferential,
\begin{equation}\label{051709}
\frac{1}{\mathbb{E}_{u\sim\mathcal{N}(0,I)}[\mathds{1}_{\mathfrak{D}_{\sigma}\cap\mathbb{B}_{M}}(u)]}
\mathbb{E}_{u\sim\mathcal{N}(0,I)}[\nabla f(x+\sigma u)\cdot\mathds{1}_{\mathfrak{D}_{\sigma}\cap\mathbb{B}_{M}}(u)]\in\partial_{G}^{\delta}f(x).
\end{equation}
The desired conclusion follows from \eqref{importantintegral}, \eqref{051709}, \eqref{hahaha} and the definition of $\Delta_{h}$ in \eqref{HB051702}.
\end{proof}	
\begin{remark}[{{Simplified expressions for the choice of $\sigma$}}]\label{sdde}
We present more explicit upper bounds for $\sigma$ in Theorem~\ref{THBD01}\ref{THBD01:ii}.
Let $f\in \spb$ with parameters ${\rm R}_1$, ${\rm R}_2$ and $m$ as in \eqref{definitioneq01}. Let
  \begin{equation}\label{varepsilonchoice}
    0<\varepsilon < \min\left\{5{\rm R_2},1\right\} \ \ {\rm and}\ \ 0< \delta < 1.
  \end{equation}
To provide an estimate on the corresponding $\bar\sigma$ in \eqref{defsigma}, we define
 \begin{align}\label{defQ}
  \mathfrak{M}_{1} &= [(2\pi)^\frac{d}2-0.5]{\cal P},\ \ \mathfrak{M}_{2} \!=\! \left(\pi e/5\right)^{\frac{d}2}\!\!{\cal P},
 \end{align}
where ${\cal P}$ is given in \eqref{PPP}. Since ${\cal P}\ge 4{\rm R}_2$, we can deduce that
 \[
 \begin{aligned}
 \mathfrak{M}_{1}\geq 5{\rm R_2}\ {\rm and}\
 \mathfrak{M}_{2}\geq\left(\pi e/5\right)^{d/2}4{\rm R_2} \ge 5{\rm R_2}.
 \end{aligned}
 \]
 This means that for the $\varepsilon$ and $\delta$ chosen as in \eqref{varepsilonchoice}, we indeed have
\begin{equation}\label{083101}
0<\varepsilon<\min\left\{\mathfrak{M}_{1},\mathfrak{M}_{2},1\right\}\ \ \text{and}\ \ 0<\delta<1.
\end{equation}
Using the definitions of $\mathfrak{M}_{1}$ and $\mathfrak{M}_{2}$ above and the definition of $\eta_1$ in \eqref{eta1}, we can then deduce that\footnote{Specifically, we deduce from $\varepsilon < {\frak M}_1$ and the definition of $\eta_1$ that $\eta_1 = \varepsilon {\cal P}^{-1}$, and then from $\varepsilon < {\frak M}_2$ and the definition of $\eta_1$ that $\eta_{1}^{\frac{2}{d}}/(2\pi e)<1/10$.}
 \begin{equation}\label{eta1eta1}
\eta_{1}=\varepsilon\mathcal{P}^{-1}\ \ \text{and}\ \ 0<\eta_{1}^{\frac{2}{d}}/(2\pi e)<1/10<1/e.
 \end{equation}
Next, recall that for $0<h<\frac{1}{10}$, it holds that
\[
W_{-1}(-h)\geq \frac{e}{e-1}\ln(h)\geq\frac{e}{e-1}\frac{0.1\ln(0.1)}{h}>-\frac{1}{2h},
\]
where the first inequality follows from \cite[Eq.~(8)]{Loczi}, and the second inequality holds because $t\mapsto t \ln t$ is decreasing on $[0,0.1]$. The above display together with \eqref{eta1eta1} implies that for the $H$ given in \eqref{defH},
\begin{equation}\label{boundH}
H\leq\sqrt{d\pi e}\eta_{1}^{-\frac{1}{d}}.
\end{equation}
Finally, since $H > \sqrt{d}$ (see footnote~\ref{footnote}) and we chose $0<\delta<1$ as stated in \eqref{varepsilonchoice}, it follows from \eqref{defsigma} that $\bar{\sigma}=\min\big\{\big[\frac{\varepsilon}{2^{m+1}{\rm R}_{1}(m+d)^{\frac{m}{2}}}\big]^{\frac{1}{m}}, \frac{\delta}{H}\big\}$. Therefore, for the $\varepsilon$ and $\delta$ chosen as in \eqref{varepsilonchoice}, {\color{revise} upon combining \eqref{083101}, \eqref{eta1eta1} and \eqref{boundH} and recalling that ${\rm R}_1 = 0$ if and only if $m=0$, we have} that the inclusion in Theorem~\ref{THBD01}\ref{THBD01:ii} holds {\color{revise}whenever
\begin{equation}\label{091001}
 \sigma\leq
\begin{cases}
 \min\big\{[2^{m+1}{\rm R_1}(m+d)^{\frac{m}{2}}]^{-\frac{1}{m}},\delta\mathcal{P}^{-1/d}/\sqrt{d\pi e}\big\}\cdot\varepsilon^{\max\{\frac{1}{m},\frac{1}{d}\}}& {\color{revise}{\rm if}\ m \ge 1,} \\
 \delta\mathcal{P}^{-1/d}\varepsilon^{\frac{1}{d}}/\sqrt{d\pi e} & \text{if } m =0.
 \end{cases}
\end{equation}
}
\end{remark}
	
{\color{revise}Let $\zeta$ be an accumulation point of the set $\{\nabla f_{\sigma}(x)\}_{\sigma > 0}$ as $\sigma \to 0^+$. Since $\partial_G^\delta f(x)$ and $(1+\|x\|^m)\epsilon\cdot \mathbb{B}$ are both compact,  $\partial_G^\delta f(x) + (1+\|x\|^m)\epsilon\cdot \mathbb{B}$ is closed. Theorem~\ref{THBD01} then implies that $\zeta\in \partial_G^\delta f(x) + (1+\|x\|^m)\epsilon\cdot \mathbb{B}$.
Using the definition of Clarke subdifferential and limiting arguments, we obtain the following corollary.}
\begin{corollary}[{{GS gradient consistency}}]
Let $f\in \spb$ and let $\nabla f_{\sigma}$ and $\partial_Cf$ be given in \eqref{GSgradient} and \eqref{def_Clarkesubdiff}, respectively. Then for each $x\in \R^d$, every accumulation point of $\{\nabla f_{\sigma}(x)\}_{\sigma > 0}$ as $\sigma \to 0^+$ belongs to $\partial_{C}f(x)$.
\end{corollary}		

\section{GS-based algorithms for minimizing SPB functions}\label{sec4}

In this section, we consider the optimization problem
\begin{equation}\label{problem}
		\begin{array}{rl}
			\displaystyle{\min_{x\in \mathbb{R}^d}} & f(x)\\
			{\rm s.t.}& x\in {\color{revise} \Omega},
		\end{array}		
\end{equation}	
where ${\color{revise} \Omega}\subseteq \mathbb{R}^d$ is a closed convex set with an easy-to-compute projection, and  $f$ is an SPB function that can only be accessed through its zeroth-order oracle (i.e., a routine for computing the function value at any prescribed point). Problem~\eqref{problem} therefore falls into the category of zeroth-order (or derivative-free) optimization problems; see, e.g., \cite{ConnSchVic09,Larson2019} and references therein for classical works and recent developments on zeroth-order optimization.

From Theorem~\ref{lem053101}, for any $x\in \mathbb{R}^d$ and $\sigma>0$, both $ f(x + \sigma u)u/\sigma$ and $ (f(x + \sigma u) - f(x))u/\sigma$ are unbiased estimators of the GS gradient $\nabla f_\sigma(x)$, where $u$ is the standard Gaussian random vector. If the smoothing parameter $\sigma $ is small, they can serve as random ascent directions for $f$ at $x$. Observing that they can be computed with one or two evaluations of $f$, it is then {\color{rerevise}tempting} to develop zeroth-order algorithms based on these gradient estimators. Indeed, there is a wealth of literature on such zeroth-order algorithms; see, e.g, \cite{NesterovGS, Maggiar2018, Bala2022, Berahas2022, Jongeneel21}. We refer to them as GS-based algorithms.
Most existing works rely on the global Lipschitz continuity of $\nabla f_\sigma$, which not only offers an obvious choice of stepsize but also greatly facilitates the convergence analysis.
One technical novelty in this paper lies in the convergence analysis of GS-based zeroth-order algorithms for SPB functions that do not possess a globally Lipschitz GS gradient.

A core idea for the design of our GS-based algorithms is as follows. For any SPB $f$, Theorem~\ref{prop:gsgradient}\ref{prop:gsgradient_ii} asserts that $\nabla f_\sigma$ is locally Lipschitz with a polynomially bounded Lipschitz modulus of order $O(\|x\|^m + 1)$ for some {\color{revise}$m\ge0$}. This naturally suggests the adaptive stepsize that scales as $(\|x\|^m + 1)^{-1}$. Based on this observation, we will develop several algorithms for different subclasses of problem~\eqref{problem} in subsequent subsections.

{\color{revise} We will develop algorithms for problem~\eqref{problem} under two different settings: the constrained convex setting (where the objective function $f$ is convex) and the unconstrained non-convex setting (where the feasible region $\Omega = \mathbb{R}^d$).
Before that,} we first establish some auxiliary lemmas, which will be useful for the algorithmic analysis.

\subsection{Auxiliary lemmas for complexity analysis}\label{sec41}
{\color{revise}

A key instrument for convergence analysis of optimization algorithms is the descent lemma, which is often proved for functions with globally Lipschitz gradients; see, e.g., \cite[Theorem~2.1.5]{Nesterov04}. We present a descent lemma of $f_\sigma$ for SPB functions $f$, whose gradients $\nabla f_\sigma$ are not necessarily globally Lipschitz.
\begin{lemma}[Descent lemma]\label{lem120901}
Let $f\in \spb$ with parameters ${\rm R}_1$, ${\rm R}_2$ and $m$ as in \eqref{definitioneq01}, $f_\sigma$ be defined in Definition~\ref{def_GSfunction}, and $\mathcal{A}$, $\mathcal{B}$, $\mathcal{C}$ be given in \eqref{eq102801}. Then
\begin{align*}
  f_{\sigma}(x)-f_{\sigma}(y)\!\le\! \langle\nabla f_{\sigma}(y),x-y\rangle \!+\! \bigg[\frac{\mathcal{A}+\mathcal{B}\|y\|^{m}}{2}\!+\!\frac{\mathcal{C}\|x-y\|^{m}}{m+2}\bigg]\|x-y\|^{2}\ \ \ \ \forall x,y\in \R^d.
\end{align*}
\end{lemma}
\begin{proof}
For any $x$, $y\in\R^d$, we have
\begin{align*}
&f_{\sigma}(x)-f_{\sigma}(y)- \langle\nabla f_{\sigma}(y),x-y\rangle=\int_{0}^{1}\langle\nabla f_{\sigma}(y+t(x-y))-\nabla f_{\sigma}(y),x-y\rangle dt\\
&\le\int_{0}^{1}\|\nabla f_{\sigma}(y+t(x-y))-\nabla f_{\sigma}(y)\| dt\cdot\|x-y\|\\
&\overset{\rm (a)}\le\int_{0}^{1}[\mathcal{A}\!+\!\mathcal{B}\|y\|^{m}\!+\!\mathcal{C}\|t(x\!-\!y)\|^{m}]t dt\cdot\|x\!-\!y\|^{2}\!=\!\bigg[\frac{\mathcal{A}+\mathcal{B}\|y\|^{m}}{2}\!+\!\frac{\mathcal{C}\|x\!-\!y\|^{m}}{m+2}\bigg]\|x\!-\!y\|^{2},
\end{align*}
where (a) follows from Theorem~\ref{prop:gsgradient}\ref{prop:gsgradient_ii}.
\end{proof}

We remark that another descent lemma for functions without a Lipschitz gradient has been formulated and studied in a very recent work~\cite{MiKhWaDeGo24} (see \cite[Definition~2.1]{MiKhWaDeGo24}) based on the notion of directional smoothness and utilized to analyze gradient descent. Using their descent lemma, they further proposed an adaptive stepsize for gradient descent, which is implicitly defined by a nonlinear equation involving the current and \emph{next} iterates and requires a root-finding procedure to compute.
In contrast, the adaptive stepsize derived from Lemma~\ref{lem120901} (see Theorems~\ref{thm:complexity_convex} and \ref{HBcomplexity01} below) only depends on the current iterate and is given by an \emph{explicit} formula.

The next lemma quantifies the approximation error of $f_\sigma$.
\begin{lemma}\label{lem121001}
  Let $f\in \spb$ with parameters ${\rm R}_1$, ${\rm R}_2$ and $m$ as in \eqref{definitioneq01} and let $f_\sigma$ be defined in Definition~\ref{def_GSfunction}. Then it holds that
  \begin{equation*}
    |f_{\sigma}(x)-f(x)| \le \mathcal{M}(x)\cdot\sigma\ \ \ \ \ \ \forall x\in \R^d,
  \end{equation*}
  where $\mathcal{M}:\R^d\to \R_+$ is the function
  \begin{equation*}
  \mathcal{M}(x):=(2^{m-1}{\rm R}_{1}\|x\|^{m}+{\rm R}_{2})\sqrt{d}+2^{m-1}{\rm R}_{1}\sigma^{m}(m+1+d)^{\frac{m+1}{2}}.
  \end{equation*}
\end{lemma}
\begin{proof}
Notice that for all $x\in \R^d$, we have
\begin{align*}\label{052219}
&|f_{\sigma}(x)-f(x)|\leq\mathbb{E}_{u\sim\mathcal{N}(0,I)}[|f(x+\sigma u)-f(x)|]\\
&\leq\mathbb{E}_{u\sim\mathcal{N}(0,I)}[(2^{m-1}{\rm R}_{1}\|x\|^{m}+2^{m-1}{\rm R}_{1}\sigma^{m}\|u\|^{m}+{\rm R}_{2})\cdot\sigma\|u\|]\\
&\leq \big[(2^{m-1}{\rm R}_{1}\|x\|^{m}+{\rm R}_{2})\sqrt{d}+2^{m-1}{\rm R}_{1}\sigma^{m}(m+1+d)^{\frac{m+1}{2}}\big]\cdot\sigma=\mathcal{M}(x)\cdot\sigma,
\end{align*}
where the second inequality follows from Lemma~\ref{prop:useful} and the last inequality follows from \cite[Lemma~1]{NesterovGS}.
\end{proof}

The next two lemmas concern the random vector $\left(\frac{f(x+\sigma u) - f(x)}{\sigma}\right)u$ with $u\sim {\cal N}(0,I)$.
 \begin{lemma}\label{lem112901}
  Let $f\in \spb$ with parameters ${\rm R}_1$, ${\rm R}_2$ and $m$ as in \eqref{definitioneq01}, $x\in \R^d$, $\sigma > 0$ and $p$ be a nonnegative integer. Define $F(u)= \frac1\sigma [f(x + \sigma u)-f(x)]u$. Then 
  \begin{align}\label{eq112804_old}
\mathbb{E}_{u\sim {\cal N}(0,I)}\big[\|F(u)\|^{p}\big]
&{\color{revise}\le {\cal H}_{(p)} (\|x\|^{mp} + 1).}
\end{align}
{\color{revise}where ${\cal H}_{(\cdot)}$ is the function such that ${\cal H}_{(0)}:= 1$, and when $p \ge 1$,
\[
{\cal H}_{(p)}\!:= 3^{p-1}\!\max\{2^{(m-1)p}\mathrm{R}_{1}^{p}(2p+d)^{p},
\mathrm{R}_{2}^{p}(2p+d)^{p}+2^{(m-1)p}\mathrm{R}_{1}^{p}\sigma^{mp}[(m+2)p+d]^{\frac{(m+2)p}2}\}.
\]
}
\vspace{-0.5 cm}
\end{lemma}
\begin{proof}
Note that \eqref{eq112804_old} holds trivially for $p = 0$. While for $p \ge 1$, from Lemma \ref{prop:useful}, for any $u\in \R^d$, we have
\[
|f(x+\sigma u)-f(x)|\le(2^{m-1}\mathrm{R}_{1}\|x\|^{m}+2^{m-1}\mathrm{R}_{1}\|\sigma u\|^{m}+\mathrm{R}_{2})\|\sigma u\|.
\]
It follows from the convexity of $\|\cdot\|^p$ that
\begin{align*}
 &\bigg|\frac{f(x+\sigma u)-f(x)}{\sigma}\bigg|^{p}\|u\|^{p}\\
 &\le\big[3^{p-1}2^{(m-1)p}\mathrm{R}_{1}^{p}\|x\|^{mp}+3^{p-1}\mathrm{R}_{2}^{p}\big]\|u\|^{2p}+3^{p-1}2^{(m-1)p}\mathrm{R}_{1}^{p}\sigma^{mp}\|u\|^{(m+2)p}.
\end{align*}
Taking the expectation on both sides of the above display with respect to $u\sim {\cal N}(0,I)$ and invoking \cite[Lemma~1]{NesterovGS} for upper bounding moments of the form $\mathbb{E}_{u\sim {\cal N}(0,I)}[\|u\|^k]$ for $k \ge 1$, we see that $\mathbb{E}_{u\sim {\cal N}(0,I)}\big[\|F(u)\|^{p}\big]$ can be bounded from above by
\[
\big[3^{p-1}2^{(m-1)p}\mathrm{R}_{1}^{p}\|x\|^{mp}+3^{p-1}\mathrm{R}_{2}^{p}\big](2p+d)^{p}+3^{p-1}2^{(m-1)p}\mathrm{R}_{1}^{p}\sigma^{mp}[(m+2)p+d]^{(m+2)p/2}.
\]
The relation \eqref{eq112804_old} now follows immediately.
\end{proof}
}

{\color{revise}
\begin{lemma}\label{lem082701}
Let $x$ be a random vector and $g$ be a nonnegative, lower semicontinuous function. Assume that
\begin{align}\label{conditionss}
 \mathbb{E}_{x}[\|x\|^{2}]\le\beta_{c}\ \ \text{and}\ \ \mathbb{E}_{x}\left[\frac{g(x)}{1+\|x\|^{n}}\right]\le\alpha_{c}
\end{align}
for some integer $n\ge1$ and positive numbers $\alpha_{c}$ and $\beta_{c}$. Then
\begin{align*}
 \mathbb{E}_{x}\left[g(x)^{\frac{1}{2\lceil n/2\rceil}}\right]\le\left(1+\sqrt{\beta_{c}}\right)(2\alpha_{c})^{\frac{1}{2\lceil n/2\rceil}}.
\end{align*}
\end{lemma}
\begin{proof}
Notice that
\begin{align}\label{eq082601}
&\mathbb{E}_{x}\Bigg[\Bigg(\frac{g(x)^{\frac{1}{2\lceil n/2\rceil}}}{1+\|x\|}\Bigg)^{2}\Bigg]
\overset{\rm (a)}\le\mathbb{E}_{x}\Bigg[\frac{g(x)^{\frac{1}{\lceil n/2\rceil}}}{1+\|x\|^{2}}\Bigg]\overset{\rm (b)}\le\Bigg(\mathbb{E}_{x}\Bigg[\Bigg(\frac{g(x)^{\frac{1}{\lceil n/2\rceil}}}{1+\|x\|^{2}}\Bigg)^{\lceil n/2\rceil}\Bigg]\Bigg)^\frac1{\lceil n/2\rceil}\notag\\
& \ \ \ \ \ \ \ \ = \Bigg(\mathbb{E}_{x}\Bigg[\frac{g(x)}{(1+\|x\|^{2})^{\lceil n/2\rceil}}\Bigg]\Bigg)^\frac1{\lceil n/2\rceil}
\!\!\overset{\rm (c)}\le\! \Bigg(2 \mathbb{E}_{x}\Bigg[\frac{g(x)}{1+\|x\|^{n}}\Bigg]\Bigg)^\frac1{\lceil n/2\rceil}
\!\!\!\le\!(2\alpha_{c})^{\frac{1}{\lceil n/2\rceil}},
\end{align}
where (a) holds because $1+\|x\|^{2}\le(1+\|x\|)^{2}$, (b) follows from the Jensen's inequality,
(c) holds since $1+\|x\|^{n}\le2(1+\|x\|^{2\lceil n/2\rceil})\le2(1+\|x\|^{2})^{\lceil n/2\rceil}$, and the last inequality follows from \eqref{conditionss}.
On the other hand, note that $\mathbb{E}_{x}[\|x\|^{2}]\le\beta_{c}$ implies $\mathbb{E}_{x}[\|x\|]\le\sqrt{\beta_{c}}$ and hence
$\mathbb{E}_{x}[(1+\|x\|)^{2}]\le(1+\sqrt{\beta_{c}})^{2}$.
Combining this with \eqref{eq082601} gives
\begin{align*}
 (1+\sqrt{\beta_{c}})^2(2\alpha_{c})^{\frac{1}{\lceil n/2\rceil}}
 \!\ge\!
\mathbb{E}_{x}[(1+\|x\|)^{2}]\cdot\mathbb{E}_{x}\Bigg[\Bigg(\frac{g(x)^{\frac{1}{2\lceil n/2\rceil}}}{1+\|x\|}\Bigg)^{2}\Bigg]
\!\ge\!
 \Bigg(\mathbb{E}_{x}\bigg[g(x)^{\frac{1}{2\lceil n/2\rceil}}\bigg]\Bigg)^2.
\end{align*}
\end{proof}
}

\subsection{\color{revise}Convex SPB minimization}\label{sec431}
Here, we assume in addition that the objective function $f$ is convex, rendering problem~\eqref{problem} a convex optimization problem. The specific algorithm for this case is {\color{revise}presented in Algorithm~\ref{algorithm 001}}. {\color{revise}Algorithm~\ref{algorithm 001} can be seen as a natural extension of the algorithm in \cite[Eq.~(39)]{NesterovGS}: indeed, when $m = 0$, our algorithm essentially reduces to their algorithm (with an extra scaling factor of $0.5$ in our stepsize). Here, we scale the stepsize by $\|x^{k}\|^{m}+1$ to account for the lack of Lipschitzness.}
\begin{algorithm}[h]
\caption{GS-based zeroth-order algorithm for convex problem~\eqref{problem}} \label{algorithm 001}
\begin{algorithmic}[1]
\State \textbf{Input:} Initial point $x^0\in {\color{revise} \Omega}$, $\{\tau_k\}\subset(0,1]$ and $\sigma>0$. Let $m$ be defined as in \eqref{definitioneq01} corresponding to our $f\in \spb$.
\For {$k = 0,1,2,\dots$}
\State Generate $u^k\sim {\cal N}(0,I)$ and form $v^k= \frac1\sigma [f(x^k + \sigma u^k)-f(x^k)]u^k$.

\State Compute
			\[
				x^{k+1} =P_{{\color{revise} \Omega}} \left(x^k - \tau_k \cdot \frac{v^{k}}{\|x^{k}\|^{m}+1}\right).
			\]	
\EndFor
\end{algorithmic}
\end{algorithm}

{\color{revise}The following result establishes the convergence rate of Algorithm~\ref{algorithm 001}.
Due to the weak assumptions, $\{x^k\}$ can be unbounded in general. Therefore, a \emph{relative} optimality measure that is rescaled by $\|x^k\|^m + 1$ is adopted.
Complexity result based on the standard measure $f(x^{k})-f(x^*)$ will be derived in Corollary~\ref{cor4.7} under additional assumptions.}

\begin{theorem}[{{Complexity bound for Algorithm~\ref{algorithm 001}}}]\label{thm:complexity_convex}
Consider problem~\eqref{problem}, where $f\in \spb$ with parameters ${\rm R}_1$, ${\rm R}_2$ and $m$ as in \eqref{definitioneq01}. Assume additionally that $f$ is convex and there exists an optimal solution $x^*$ for \eqref{problem}.
Then the sequence $\{x^k\}$ generated by Algorithm~\ref{algorithm 001} satisfies that
\begin{align}
  {\color{revise}0  \le \mathbb{E}[\|x^{k}-z\|^{2}]-\mathbb{E}[\|x^{k+1}-z\|^{2}]+{\cal H}_{(2)}\tau_{k}^{2}\ \ \ \ \ \ \ \forall z\in\argmin_{u\in {\color{revise} \Omega}}f_\sigma(u),}\label{eq090502}
\end{align}
and for any $T > 0$,
\begin{align}
\min_{0\le k\le T}\mathbb{E}\bigg[\frac{f(x^{k})-f(x^*)}{\|x^{k}\|^{m}+1}\bigg] \!\le\! \frac{1}{2\sum_{k=0}^{T}\tau_{k}} \bigg[\|x^{0}-x^*\|^{2}\!+\!{\color{revise}{\cal H}_{(2)}}\sum_{k=0}^{T}\tau_{k}^{2}\bigg]\!+\!\mathcal{M}(x^*)\cdot \sigma,\label{hahaheheha}
\end{align}
where {\color{revise}$\mathcal{M}(\cdot)$ and ${\cal H}_{(\cdot)}$ are defined in Lemma~\ref{lem121001} and Lemma~\ref{lem112901}, respectively.}
\end{theorem}
\begin{proof}
Since the projection operator is nonexpansive, we see that {\color{revise}for any $z\in \Omega$,}
\begin{align*}
  {\color{revise}\|x^{k+1}-z\|^2} & {\color{revise}=\|P_{{\color{revise} \Omega}}(x^{k}-\alpha_{k}v^{k})-z\|^{2}\le\|x^{k}-z-\alpha_{k}v^{k}\|^{2}}\\
  &{\color{revise}=\|x^{k}-z\|^{2}-2\alpha_{k}\langle v^{k},x^{k}-z\rangle+\alpha_{k}^{2}\|v^{k}\|^{2},}
\end{align*}
where $\alpha_{k}=\tau_{k}/(\|x^{k}\|^{m}+1)$. Taking the expectation on both sides of the above inequality with respect to the random variable $v^k$, we can obtain from \eqref{GSgradient} that
\begin{align}\label{display1}
{\color{revise}2\alpha_{k}\langle \nabla f_{\sigma}(x^{k}),x^{k}\!-\!z\rangle\!\le \!\|x^{k}\!\!-\!\!z\|^{2}\!-\!\mathbb{E}_{u^k\sim {\cal N}(0,I)}[\|x^{k+1}\!\!-\!z\|^{2}]\!+\!\alpha_{k}^{2}\mathbb{E}_{u^k\sim {\cal N}(0,I)}[\|v^{k}\|^{2}].\!\!\!\!\!\!\!}
\end{align}
Next, by Definition~\ref{def_GSfunction} and the convexity of $f$, $f_{\sigma}$ is also convex. Then we have the following subgradient inequality:
\begin{align}\label{display2}
  {\color{revise}f_{\sigma}(x^{k})-f_{\sigma}(z)  \le \langle \nabla f_{\sigma}(x^{k}),x^{k}-z\rangle.}
\end{align}
In addition, we have from {\color{revise}Lemma~\ref{lem112901} with $p = 2$} that
\begin{align}\label{display3}
 \mathbb{E}_{u^k\sim {\cal N}(0,I)}\left[\|v^{k}\|^{2}\right]\le {\color{revise}{\cal H}_{(2)}}(\|x^k\|^{2m}+1)\le{\color{revise}{\cal H}_{(2)}}(\|x^k\|^{m}+1)^{2}.
\end{align}
Combining \eqref{display1}, \eqref{display2} and \eqref{display3} yields
\begin{align}\label{display3.5}
  {\color{revise} 2\alpha_{k}[f_{\sigma}(x^{k})-f_{\sigma}(z)]  \le \|x^{k}-z\|^{2}-\mathbb{E}_{u^{k}\sim {\cal N}(0,I)}[\|x^{k+1}-z\|^{2}]+{\cal H}_{(2)}\tau_{k}^{2},}
\end{align}
{\color{revise} which gives \eqref{eq090502} upon taking expectation on both sides.}

We now lower bound the left-hand side of \eqref{display3.5} {\color{revise}for $z=x^*$}. To this end, we first note from \cite[Eq.~(11)]{NesterovGS} that
\begin{align}\label{display4}
  f_{\sigma}(x^{k}) \geq f(x^{k}).
\end{align}
Also, we have
\begin{align}\label{display5}
  f_{\sigma}(x^*) =f_{\sigma}(x^*)-f(x^*)+f(x^*)\le \mathcal{M}(x^*)\cdot\sigma+f(x^*),
\end{align}
where the inequality follows from Lemma~\ref{lem121001}. Using \eqref{display4} and \eqref{display5}, we can lower bound the left hand side of {\color{revise}\eqref{eq090502}}, which in turn yields
\begin{align*}
  \mathbb{E}[2\alpha_{k}(f(x^{k})-f(x^*))] \le \mathbb{E}[\|x^{k}-x^*\|^{2}]-\mathbb{E}[\|x^{k+1}-x^*\|^{2}]+2\mathcal{M}(x^*)\sigma\cdot\alpha_{k}+{\color{revise}{\cal H}_{(2)}}\tau_{k}^{2}.
\end{align*}
{\color{revise}Finally,} invoking the definition of $\alpha_k$ (and noting also that $\alpha_k\le \tau_k$), we deduce {\color{revise} from the above display} that
\begin{align*}
  2\tau_{k}\mathbb{E}\!\left[\frac{f(x^{k})-f(x^*)}{\|x^{k}\|^{m}+1}\right] \le \mathbb{E}\left[\|x^{k}-x^*\|^{2}\right] - \mathbb{E}[\|x^{k+1}-x^*\|^{2}] + 2\mathcal{M}(x^*)\sigma\cdot\tau_{k} + {\color{revise}{\cal H}_{(2)}}\tau_{k}^{2},
\end{align*}
which, upon summing over $k$, completes the proof.
\end{proof}
{\color{revise}
\begin{remark}[Comparing existing complexity results for \eqref{problem} with a convex $f$]
When $m = 0$, ${\rm R}_1 = 0$. We then see from Lemma~\ref{lem112901} that ${\cal H}_{(2)} = 3{\rm R}_2^2(4+d)^2$ and Lemma~\ref{lem121001} that $\mathcal{M}(x^*)= {\rm R}_2 \sqrt{d}$. These together with \eqref{hahaheheha} give a bound that matches the one obtained in \cite[Theorem~6]{NesterovGS} up to a constant scaling factor.
\end{remark}
}

{\color{revise}
Theorem~\ref{thm:complexity_convex} gives a bound on $\mathbb{E}\bigg[\frac{f(x^{k})-f(x^*)}{\|x^{k}\|^{m}+1}\bigg]$, which can be regarded as a relative optimality measure when $m\ge 1$. The next corollary shows that under suitable assumptions on $f$, one can derive a bound on an \emph{absolute} optimality measure.
}

{\color{revise}
\begin{corollary}\label{cor4.7}
Consider problem~\eqref{problem}, where $f\in \spb$ with parameters ${\rm R}_1$, ${\rm R}_2$ and $m\ge1$ in \eqref{definitioneq01}. Assume that $f$ is convex and level-bounded, and let $x^*$ be a minimizer for \eqref{problem}. Let $\gamma \in (0,1]$ and $T$ be a positive integer. If $\tau_{k}={\gamma}/{\sqrt{T+1}}$ for $k=0,\cdots,T$, then the sequence $\{x^k\}$ generated by Algorithm~\ref{algorithm 001} satisfies
\begin{align*}
\min_{0\le k\le T}\mathbb{E}\big[(f(x^{k})-f(x^*))^{\frac{1}{(2\lceil m/2\rceil)}}\big]
\le\big(1+\sqrt{M_{\rm bd}}\big)\big(2C_{\rm bd}T^{-1/2}+2\mathcal{M}(x^*)\sigma\big)^{\frac{1}{(2\lceil m/2\rceil)}},
\end{align*}
where
\begin{align}
 M_{\rm bd}&=4\|x^{0}\|^{2}+6C_{\rm lev}^{2}+2{\cal H}_{(2)}\gamma^{2}, \ \ \ C_{\rm bd}=\frac{1}{2\gamma} \left(\|x^{0}-x^*\|^{2}+{\cal H}_{(2)}\gamma^{2}\right),\label{mdcd2}\\
 C_{\rm lev}&=\sup\{\|x\|:f(x) \le f(x^*)+\mathcal{M}(x^*)\sigma\} < \infty, \label{alphastar}
\end{align}
$\mathcal{M}(\cdot)$ is defined in Lemma~\ref{lem121001} and $ {\cal H}_{(.)}$ is defined in Lemma~\ref{lem112901}.\footnote{\color{revise}The finiteness of $C_{\rm lev}$ follows from the assumption that $f$ is level-bounded.}
\end{corollary}

\begin{proof}
Let $k_{*}\in\arg\min_{0\le  k \le T}\mathbb{E}\left[\frac{f(x^{k})-f(x^*)}{\|x^{k}\|^{m}+1}\right]$. From Theorem \ref{thm:complexity_convex}, we have
\begin{align}\label{eq082702pm}
&\mathbb{E}\left[\frac{f(x^{k_{*}})-f(x^*)}{\|x^{k_{*}}\|^{m}+1}\right]
\le \frac{1}{2\sum_{k=0}^{T}\tau_{k}} \left[\|x^{0}-x^*\|^{2}+{\cal H}_{(2)}\sum_{k=0}^{T}\tau_{k}^{2}\right]+\mathcal{M}(x^*)\sigma\notag\\
&\overset{\rm (a)}\le \frac{1}{2\gamma\sqrt{T}} \left(\|x^{0}-x^*\|^{2}+{\cal H}_{(2)}\gamma^{2}\right)+\mathcal{M}(x^*)\sigma \le\frac{C_{\rm bd}}{\sqrt{T}}+\mathcal{M}(x^*)\sigma,
\end{align}
where we used the fact that $\tau_k = {\gamma}/{\sqrt{T+1}}$ in (a), and $C_{\rm bd}$ is defined in \eqref{mdcd2}.

Next, pick any $z^*\in \arg\min_{x\in {\color{revise} \Omega}}f_{\sigma}(x)$.\footnote{\color{revise}Notice that $\arg\min_{x\in {\color{revise} \Omega}}f_\sigma(x)$ is nonempty because the level-boundedness of $f$ and \cite[Eq.~(11)]{NesterovGS} imply the level-boundedness of $f_\sigma$.} Then we see from \eqref{eq090502} that
\begin{align}\label{nonnegative}
  0\le \mathbb{E}[\|x^{k}-z^*\|^{2}]-\mathbb{E}[\|x^{k+1}-z^*\|^{2}]+{\cal H}_{(2)}\tau_{k}^{2}.
\end{align}
When $k_* \ge 1$, we can sum both sides of \eqref{nonnegative} from $k = 0$ to $k_*-1$ to obtain
\begin{align*}
 \mathbb{E}[\|x^{k_{*}}-z^*\|^{2}]&\le\|x^{0}-z^*\|^{2}+{\cal H}_{(2)}\sum_{i=0}^{k_{*}-1}\tau_{i}^{2}
 \le2\|x^{0}\|^{2}+2\|z^*\|^{2}+{\cal H}_{(2)}\gamma^{2},
\end{align*}
where we used the fact that $\tau_k = {\gamma}/{\sqrt{T+1}}$ for the last inequality.
The above display further implies that
\begin{align}\label{eq082701pm0}
 \mathbb{E}[\|x^{k_{*}}\|^{2}]\le 2\mathbb{E}[\|x^{k_{*}}-z^*\|^{2}] + 2\|z^*\|^2 \le4\|x^{0}\|^{2}+6\|z^*\|^{2}+2{\cal H}_{(2)}\gamma^{2}.
\end{align}
Note that the above inequality also holds when $k_* = 0$.

Now, since $z^*\in \arg\min_{x\in \Omega}f_{\sigma}(x)$, we have
\begin{align*}
f(z^*)\le f_{\sigma}(z^*)\le f_{\sigma}(x^*)\le f(x^*)+\mathcal{M}(x^*)\sigma,
\end{align*}
where the first inequality follows from \cite[Eq.~(11)]{NesterovGS} and the last inequality follows from Lemma~\ref{lem121001}. Thus, it holds that $\sup\{\|z\|:\; z\in \arg\min_{x\in \Omega}f_{\sigma}(x)\}\le C_{\rm lev}<+\infty$, where $C_{\rm lev}$ is defined in \eqref{alphastar}. This observation together with \eqref{eq082701pm0} implies that
\begin{align}\label{eq082701pm}
 \mathbb{E}[\|x^{k_{*}}\|^{2}]\le M_{\rm bd},
\end{align}
where $M_{\rm bd}$ is defined in \eqref{mdcd2}.
The desired conclusion now follows immediately upon combining \eqref{eq082701pm} with \eqref{eq082702pm} and Lemma \ref{lem082701}.
\end{proof}
}

\subsection{Unconstrained SPB minimization}
Here, we consider problem~\eqref{problem} with ${\color{revise} \Omega} = \R^d$. The specific algorithm is presented in Algorithm~\ref{algorithm 01} below.
Notice that the update rule for $x^k$ differs from that of {\color{revise}Algorithm~\ref{algorithm 001} in that the stepsize} has to be rescaled by $\|x^{k}\|^{2m}+1$ instead of $\|x^k\|^m + 1$.
This rescaling also makes our algorithm (for SPB functions) {\color{revise}a natural extension of the ones in \cite[Eq.~(66)]{NesterovGS} and \cite[Algorithm~1]{Lin2022}}, which are designed for $f$ being globally Lipschitz. {\color{revise} Indeed, when $m = 0$, our algorithm essentially reduces to \cite[Eq.~(66)]{NesterovGS} (with an extra factor of $0.5$ in our stepsize).}

\begin{algorithm}[h]
\caption{GS-based zeroth-order algorithm for unconstrained problem~\eqref{problem}} \label{algorithm 01}
\begin{algorithmic}[1]
\State \textbf{Input:} Initial point $x^0\in \mathbb{R}^d$, $\{\tau_k\}\subset(0,1]$ and $\sigma>0$. Let $m$ be defined as in \eqref{definitioneq01} corresponding to our $f\in \spb$.
\For {$k = 0,1,2,\dots$}
\State Generate $u^k\sim {\cal N}(0,I)$ and form $v^k = \frac1\sigma [f(x^k + \sigma u^k)-f(x^k)]u^k$.

\State Compute
			\[
				x^{k+1} = x^k - \tau_k \cdot \frac{v^{k}}{\|x^{k}\|^{2m}+1}.
			\]
\EndFor
\end{algorithmic}
\end{algorithm}
%
%

\begin{theorem}[{{Complexity bound for Algorithm~\ref{algorithm 01}}}]\label{HBcomplexity01}
Consider \eqref{problem}, where $f\in \spb$ with parameters ${\rm R}_1$, ${\rm R}_2$ and $m$ as in \eqref{definitioneq01}. Let $\nabla f_{\sigma}$ be given in \eqref{GSgradient} and $\mathcal{A}$, $\mathcal{B}$, $\mathcal{C}$ be given in \eqref{eq102801}. Assume in addition that ${\color{revise} \Omega} = \mathbb{R}^d$ and {\color{revise}$\inf f > -\infty$}. Then the sequence $\{x^k\}$ generated by Algorithm~\ref{algorithm 01} satisfies that
\begin{align}\label{hehehaha1}
{\color{revise}
\mathbb{E}\left[f_{\sigma}(x^{k+1})\right]\le
\mathbb{E}\left[f_{\sigma}(x^{k})\right]
  +\frac{{\color{revise}{\cal H}_{(2)}}\tau_{k}^{2}({\cal A}+{\cal B})}2+\frac{\mathcal{C}}{m+2}{\color{revise}{\cal H}_{(m+2)}}\tau_{k}^{m+2},
}
\end{align}
{\color{revise}and} for any $T > 0$,
\begin{align}
{\color{revise}\min_{0\le k\le T}\widetilde{w}_{k}^{2}\!\le\! \frac{\mathcal{M}(x^0)\sigma\!+\!f(x^{0})\!-\!\inf f\!+\!0.5{\cal H}_{(2)}({\cal A}+{\cal B})\sum_{k=0}^{T}\tau_{k}^{2}+\frac{{\cal H}_{(m+2)} \mathcal{C}}{m+2}\sum_{k=0}^{T}\tau_{k}^{m+2}}{\sum_{k=0}^T\tau_k}},\notag
\end{align}
where {\color{revise}$\widetilde{w}_{k}^{2}=\mathbb{E}\big[\big\|\frac{\nabla f_{\sigma}(x^{k})}{\|x^{k}\|^{m}+1}\big\|^{2}\big]$}, $\mathcal{M}(\cdot)$ is defined in Lemma~\ref{lem121001} and {\color{revise} ${\cal H}_{(\cdot)}$ is defined in Lemma~\ref{lem112901}}.
\end{theorem}
\begin{proof}
  From Lemma \ref{lem120901}, we know that for any $k \ge 0$,
\begin{align}
f_{\sigma}(x^{k+1})&\leq f_{\sigma}(x^{k})-\langle\nabla f_{\sigma}(x^{k}),x^{k}-x^{k+1}\rangle\notag\\
&~~~+\frac{1}{2}\|x^{k}-x^{k+1}\|^{2}\bigg[\mathcal{A}+\mathcal{B}\|x^{k}\|^{m}+\frac{2\mathcal{C}\|x^{k+1}-x^{k}\|^{m}}{m+2}\bigg].\notag
\end{align}
This implies that
\begin{align*}
&\widetilde\alpha_{k}\langle\nabla f_{\sigma}(x^k),v^k\rangle\le f_{\sigma}(x^k)-f_{\sigma}(x^{k+1})+\frac{(\mathcal{A}+\mathcal{B}\|x^k\|^{m})\widetilde\alpha_{k}^{2}}2\|v^{k}\|^{2}
+\frac{\mathcal{C}\widetilde\alpha_{k}^{m+2}}{m+2}\|v^{k}\|^{m+2},
\end{align*}
where
\begin{equation}\label{alphak01}
  \widetilde\alpha_k = \tau_k / (\|x^k\|^{2m} + 1).
\end{equation}
By taking the expectation on both sides of the above inequality, we can obtain from \eqref{GSgradient} that
\begin{align}
  \!\widetilde\alpha_{k}\|\nabla f_{\sigma}(x^k)\|^{2}
  \!&\!\le f_{\sigma}(x^k)-\mathbb{E}_{u^k\sim {\cal N}(0,I)}\left[ f_{\sigma}(x^{k+1})\right]\notag\\
  \!&~~+\!\frac{(\mathcal{A}+\mathcal{B}\|x^k\|^{m})\widetilde\alpha_{k}^{2}}2\mathbb{E}_{u^k\sim {\cal N}(0,I)}[\|v^{k}\|^{2}]
\!+\!\frac{\mathcal{C}\widetilde\alpha_{k}^{m+2}}{m+2}\mathbb{E}_{u^k\sim {\cal N}(0,I)}[\|v^{k}\|^{m+2}]\label{eq121001}\!.
\end{align}

We now upper bound the two terms $(\mathcal{A}+\mathcal{B}\|x^k\|^{m})\widetilde\alpha_{k}^{2}\mathbb{E}_{u^k\sim {\cal N}(0,I)}[\|v^{k}\|^{2}]$ and $\widetilde\alpha_{k}^{m+2}\mathbb{E}_{u^k\sim {\cal N}(0,I)}\left[\|v^{k}\|^{m+2}\right]$ in \eqref{eq121001}.
For the former term, we have
\begin{align*}
 &(\mathcal{A}+\mathcal{B}\|x^k\|^{m})\widetilde\alpha_{k}^{2}\mathbb{E}_{u^k\sim {\cal N}(0,I)}[\|v^{k}\|^{2}]
 \le(\mathcal{A}+\mathcal{B}\|x^k\|^{m})\widetilde\alpha_{k}^{2}{\color{revise}{\cal H}_{(2)}}(\|x^k\|^{2m}+1)\\
 &\le(\mathcal{A}+\mathcal{B}+\mathcal{B}\|x^k\|^{2m})\widetilde\alpha_{k}^{2}{\color{revise}{\cal H}_{(2)}}(\|x^k\|^{2m}+1)\\
 &\le{\color{revise}{\cal H}_{(2)}}\cdot({\cal A}+{\cal B})\cdot(\|x^k\|^{2m}+1)^{2}\widetilde\alpha_{k}^{2} ={\color{revise}{\cal H}_{(2)}}\tau_{k}^{2}({\cal A}+{\cal B}),
\end{align*}
where {\color{revise}the first inequality} follows from {\color{revise}Lemma~\ref{lem112901} with $p = 2$}, and we used the definition of $\widetilde\alpha_k$ in \eqref{alphak01} for the equality. As for the latter term (i.e.,
$\widetilde\alpha_{k}^{m+2}\mathbb{E}_{u^k\sim {\cal N}(0,I)}[\|v^{k}\|^{m+2}]$), we can deduce from {\color{revise}Lemma~\ref{lem112901} with $p = m+2$} that
\begin{align*}
  \widetilde\alpha_{k}^{m+2}\mathbb{E}_{u^k\sim {\cal N}(0,I)}[\|v^{k}\|^{m+2}] & \le\widetilde\alpha_{k}^{m+2}
  {\color{revise}{\cal H}_{(m+2)}}[\|x^{k}\|^{m(m+2)}+1]\\
  &\le\widetilde\alpha_{k}^{m+2}{\color{revise}{\cal H}_{(m+2)}}[\|x^{k}\|^{2m}+1]^{(m+2)/2} \le {\color{revise}{\cal H}_{(m+2)}}\tau_{k}^{m+2}.
\end{align*}

Combining \eqref{eq121001} with the above two displays, one has
\begin{align*}
  &\tau_{k}\mathbb{E}\bigg[\bigg\|\frac{\nabla f_{\sigma}(x^{k})}{\|x^{k}\|^{m}+1}\bigg\|^{2}\bigg] \overset{\rm (a)}\le \mathbb{E}[\widetilde\alpha_{k}\|\nabla f_{\sigma}(x^{k})\|^{2}]\\
  &\le\mathbb{E}[f_{\sigma}(x^{k})]-\mathbb{E}[f_{\sigma}(x^{k+1})]
  +\frac{{\color{revise}{\cal H}_{(2)}}\tau_{k}^{2}({\cal A}+{\cal B})}2+\frac{\mathcal{C}}{m+2}{\color{revise}{\cal H}_{(m+2)}}\tau_{k}^{m+2},
\end{align*}
where (a) follows from \eqref{alphak01}. {\color{revise}This proves \eqref{hehehaha1}.}
Summing both sides of the above display from $k = 0$ to $T$, we obtain further that
\begin{align*}
\sum_{k=0}^{T}\tau_{k}{\color{revise}\widetilde{w}_{k}^{2}}\le f_{\sigma}(x^{0})-\mathbb{E}[f_{\sigma}(x^{T+1})]+0.5
{\color{revise}{\cal H}_{(2)}}\cdot({\cal A}+{\cal B})\sum_{k=0}^{T}\tau_{k}^{2}+\frac{{\color{revise}{\cal H}_{(m+2)}}\cdot \mathcal{C}}{m+2}\sum_{k=0}^{T}\tau_{k}^{m+2}.
\end{align*}
Finally, we
{\color{revise}have
  $f_{\sigma}(x^{0}) \le \mathcal{M}(x^0)\cdot\sigma+f(x^{0})$ from Lemma \ref{lem121001}.}
The desired result now follows immediately upon combining this last observation with the above display.
\end{proof}
{\color{revise}
\begin{remark}[Comparing existing complexity results for \eqref{problem} with ${\color{revise} \Omega} = \R^d$]
When $m=0$, $\mathrm{R}_{1}=0$. We then see from \eqref{eq102801} that ${\cal A} = {\rm R}_2\sqrt{d}/\sigma$, ${\cal B} = {\cal C} = 0$. Moreover, we have $\mathcal{M}(x^0) = {\rm  R}_2\sqrt{d}$ (see Lemma~\ref{lem121001}), ${\cal H}_{(m+2)} = {\cal H}_{(2)} = 3{\rm  R}_2^2(4+d)^2$ (see Lemma~\ref{lem112901}).
If we let $\tau_{k}\equiv\tau$ for some $\tau \in (0,1]$, we see from Theorem \ref{HBcomplexity01} that
  \begin{align}\label{eq082001}
    \min_{0\le k\le T}\widetilde{w}_{k}^{2}\le \frac{1}{\tau (T+1)}\left[\Delta+\mathrm{R}_{2}\sqrt{d}\sigma+\frac{3\mathrm{R}_{2}^{3}}{2\sigma}(4+d)^{2}\sqrt{d}(T+1)\tau^{2}\right],
  \end{align}
  where $\Delta:=f(x^{0})-\inf f$. Let $\delta_{a} = \mathrm{R}_{2}\sqrt{d}\sigma$.\footnote{\color{revise}Note that in view of Lemma~\ref{lem121001} and \cite[Theorem~1]{NesterovGS}, our definition of $\delta_a$ corresponds to the $\epsilon$ defined three lines below \cite[Eq~(69)]{NesterovGS}.} Then \eqref{eq082001} gives
  \begin{align}
    \min_{0\le k\le T}\mathbb{E}[\|\nabla f_{\sigma}(x^{k})\|^{2}] \le
    4\bigg[\frac{1}{\tau (T+1)}(\Delta+\delta_{a})+\frac{3\mathrm{R}_{2}^{4}}{2\delta_{a}}(4+d)^{2}d\tau\bigg],
  \end{align}
  which matches the bound {\color{rerevise}in \cite[Section 7]{NesterovGS}} (up to a constant scaling factor).
\end{remark}
}
{\color{revise}
\begin{corollary}\label{coro082801}
Consider problem~\eqref{problem}, where $f\in \spb$ with parameters ${\rm R}_1$, ${\rm R}_2$ and $m\ge1$ in \eqref{definitioneq01}, and assume that ${\color{revise} \Omega} = \R^d$. Let $\nabla f_{\sigma}$ be given in \eqref{GSgradient} and $\mathcal{A}$, $\mathcal{B}$, $\mathcal{C}$ be given in \eqref{eq102801}. Assume that $\mathcal{S}:=\argmin_{u\in \R^d}f(u)$ is nonempty and bounded, and there exists $\mu > 0$ such that
\begin{align}\label{quad_growth2}
  f(x)-\inf_{u\in \R^d}f(u)\ge\frac{\mu}{2}{\rm dist}\left(x, {\cal S}\right)^2 \ \ \forall x\in \R^d.
\end{align}
Let $\gamma \in (0,1]$ and $T$ be a positive integer. If $\tau_{k}=\gamma/{\sqrt{T+1}}$ for $k=0,\cdots,T$, then the sequence $\{x^k\}$ generated by Algorithm~\ref{algorithm 01} satisfies $\mathbb{E}[\|x^{k_{*}}\|^{2}]\le \widetilde{M}_{{\color{revise} \Omega}}$ for any $k_{*}\in\arg\min_{0\le  k \le T}\widetilde{w}_{k}^{2}$ and
\begin{align}
\min_{0\le k \le T}\mathbb{E}\big[\|\nabla f_{\sigma}(x^{k})\|^{\frac{1}{2\lceil m/2 \rceil}}\big]
\le2^{\frac{1}{2\lceil m/2 \rceil}}\Big(1+\sqrt{\widetilde{M}_{{\color{revise} \Omega}}}\Big)
\widetilde{C}_{{\color{revise} \Omega}}^{\frac{1}{4\lceil m/2 \rceil}}T^{-\frac{1}{8\lceil m/2\rceil}},
\end{align}
where
\begin{align}
\widetilde{C}_{{\color{revise} \Omega}}&\!=\!\frac{1}{\gamma} \bigg[f(x^{0})\!-\!\inf_{u\in \R^d}f(u) \!+\! \mathcal{M}(x^0)\sigma+\frac{{\cal H}_{(2)}({\cal A}+{\cal B})}{2}\gamma^{2}+\frac{{\cal H}_{(m+2)} \mathcal{C}}{m+2}\gamma^{m+2}\bigg],\label{cdd}\\
\widetilde{M}_{{\color{revise} \Omega}}&\!=\!8\mu^{-1}[\gamma \widetilde{C}_{{\color{revise} \Omega}} + 0.5 \mu\sigma^2 d]+2\sup_{w\in \mathcal{S}}\|w\|^2,\label{mdd}
\end{align}
$\widetilde{w}_{k}^{2}$ is defined in Theorem~\ref{HBcomplexity01},
$\mathcal{M}(\cdot)$ is defined in Lemma~\ref{lem121001} and $ {\cal H}_{(.)}$ is defined in Lemma~\ref{lem112901}.
\end{corollary}
}

{\color{revise}
\begin{remark}
  The condition \eqref{quad_growth2} is known as the second-order growth condition for $f$, and is a commonly used condition for deriving (global) asymptotic convergence rates of first-order methods; see, e.g., \cite{DimaLewis21}. It is known to hold if $f$ is strongly convex, and we refer the readers to \cite{DimaLewis21} and references therein for more concrete examples.
\end{remark}
}

{\color{revise}
\begin{proof}[Proof of Corollary~\ref{coro082801}]
We can deduce from Theorem \ref{HBcomplexity01} and $\tau_{k}\equiv \gamma/{\sqrt{T+1}}$ that
$\widetilde{w}_{k_*}^{2} \le {\widetilde{C}_{{\color{revise} \Omega}}}/{\sqrt{T}}$,
where $\widetilde{C}_{{\color{revise} \Omega}}$ is defined in \eqref{cdd}, which implies
\begin{align}
 \mathbb{E}\bigg[\bigg\|\frac{\nabla f_{\sigma}(x^{k_{*}})}{1+\|x^{k_{*}}\|^{m}}\bigg\|^{2}\bigg]\le \frac{\widetilde{C}_{{\color{revise} \Omega}}}{\sqrt{T}},\ \ \text{and hence}\ \ \mathbb{E}\bigg[\frac{\|\nabla f_{\sigma}(x^{k_{*}})\|}{1+\|x^{k_{*}}\|^{m}}\bigg]\le\widetilde{C}_{{\color{revise} \Omega}}^{\frac{1}{2}}T^{-\frac{1}{4}}.\label{eq082801pm}
\end{align}

Now, in view of \eqref{eq082801pm} and Lemma~\ref{lem082701}, the conclusion follows once we show
\begin{align}
 \mathbb{E}[\|x^{k_{*}}\|^{2}]\le \widetilde{M}_{{\color{revise} \Omega}},\label{eq082801am}
\end{align}
where $\widetilde{M}_{{\color{revise} \Omega}}$ is defined in \eqref{mdd}. Thus, in what follows, we will prove \eqref{eq082801am}.

When $k_* \ge 1$, we can sum both sides of \eqref{hehehaha1} from $k = 0$ to $k_* - 1$ to obtain
\begin{align}
\mathbb{E}[f_{\sigma}(x^{k_*})]&\le
\mathbb{E}[f_{\sigma}(x^{0})]
  +0.5{\cal H}_{(2)}({\cal A}\!+\!{\cal B})\sum_{k=0}^{k_{*}-1}\tau_{k}^{2}\!+\!\frac{{\cal H}_{(m+2)} \mathcal{C}}{m+2}\sum_{k=0}^{k_{*}-1}\tau_{k}^{m+2}\notag\\
  &\le f(x^{0})+\mathcal{M}(x^{0})\sigma+0.5{\cal H}_{(2)}({\cal A}+{\cal B})\gamma^{2}+\frac{{\cal H}_{(m+2)}\mathcal{C}}{m+2}\gamma^{m+2},\label{eq082801}
\end{align}
where the last inequality follows from Lemma~\ref{lem121001} and the fact that $\tau_k = \gamma/{\sqrt{T+1}}$.
Notice that \eqref{eq082801} also holds when $k_* = 0$.

On the other hand, according to \eqref{quad_growth2}, we can obtain
\begin{align*}
  &f_{\sigma}(x^{k_{*}})-\inf f=\mathbb{E}_{u\sim\mathcal{N}(0,I)}[f(x^{k_{*}}+\sigma u)-\inf f]\ge\frac{\mu}{2}\mathbb{E}_{u\sim\mathcal{N}(0,I)}[\mathrm{dist}(x^{k_{*}}+\sigma u, \mathcal{S})^{2}]\\
  &=\frac{\mu}{4}\mathbb{E}_{u\sim\mathcal{N}(0,I)}[2\mathrm{dist}(x^{k_{*}}+\sigma u, \mathcal{S})^{2}+2\|\sigma u\|^{2}-2\|\sigma u\|^{2}]\ge\frac{\mu}{4}{\rm dist}(x^{k_*},{\cal S})^2-\frac{1}{2}\mu d\sigma^{2},
\end{align*}
where we used \cite[Lemma 1]{NesterovGS} in the last inequality. This implies that
\begin{align}
 \mathbb{E}[f_{\sigma}(x^{k_{*}})]\ge\frac{\mu}{4}\mathbb{E}[{\rm dist}(x^{k_*},{\cal S})^2]-\frac{1}{2}\mu d\sigma^{2}+\inf f.\label{eq082802}
\end{align}
Combining \eqref{eq082801} and \eqref{eq082802}, one has
\begin{align*}
 \frac{\mu}{4}\mathbb{E}[{\rm dist}(x^{k_*},{\cal S})^2]\!\le\! f(x^{0})\!+\!\mathcal{M}(x^{0})\sigma\!+\!\frac{\mu d\sigma^{2}}2\!-\!\inf f\!+\!\frac{{\cal H}_{(2)}({\cal A}+{\cal B})}2\gamma^{2}\!+\!\frac{{\cal H}_{(m+2)}\mathcal{C}}{m+2}\gamma^{m+2}.
\end{align*}
Thus, we can deduce from the above display and the definition of $\widetilde {M}_{\Omega}$ in \eqref{mdd} that
\begin{align*}
 \mathbb{E}[\|x^{k_{*}}\|^{2}]\le 2\mathbb{E}[{\rm dist}(x^{k_*},{\cal S})^2] + 2\sup_{w\in {\cal S}}\|w\|^2\le \widetilde {M}_{\Omega}.
\end{align*}
\end{proof}
}

\section{Explicit complexity and ($\delta$, $\epsilon$)-stationarity}\label{sec44}

{\color{revise} Note that the results Theorem~\ref{HBcomplexity01} and Corollary~\ref{coro082801} for the unconstrained non-convex setting are both with respect to the GS $f_\sigma$ but not the objective function $f$. In this section, we study the iteration complexity of Algorithm~\ref{algorithm 01} for achieving a $(\delta, \epsilon)$-stationary point of an SPB function $f$.

We start by defining analogues of ${\cal H}_{(p)}$ in Lemma~\ref{lem112901} and ${\cal M}(\cdot)$ in Lemma~\ref{lem121001}:
\begin{align}
\breve{\cal H}_{(p)}&=\begin{cases}
  3^{p-1}\big(\mathrm{R}_{2}^{p}(2p+d)^{p}+2^{(m-1)p}\mathrm{R}_{1}^{p}[(m+2)p+d]^{\frac{(m+2)p}2}\big) & {\rm if}\ p \ge 1,\\
  1 & {\rm if}\ p = 0,
\end{cases}\label{breveH}\\
  \breve{\mathcal{M}}(x)&=(2^{m-1}{\rm R}_{1}\|x\|^{m}+{\rm R}_{2})\sqrt{d}+2^{m-1}{\rm R}_{1}(m+1+d)^{\frac{m+1}{2}}.\label{breveM}
\end{align}
Notice that if $\sigma$ is bounded by $1$, then it holds that
\begin{equation}\label{boundbound}
  {\cal H}_{(p)}\le \breve{\cal H}_{(p)}\ \ \forall p\in \mathbb{N}\cup \{0\}\ \  {\rm and}\ \ {\cal M}(x)\le \breve{\mathcal{M}}(x)\ \ \forall x\in\mathbb{R}^{d}.
\end{equation}
In the next auxiliary lemma, we derive bounds that \emph{explicitly} depend on $\sigma$ for the $\widetilde C_{{\color{revise} \Omega}}$ in \eqref{cdd} that appeared in Corollary~\ref{coro082801}.
\begin{lemma}\label{bound1130}
Consider problem~\eqref{problem}, where $f\in \spb$ with parameters ${\rm R}_1$, ${\rm R}_2$ and $m$ as in \eqref{definitioneq01}.
Suppose that $\inf_{u\in \R^d} f(u)>-\infty$ and let $\sigma\in(0,1]$ and $\gamma=\sigma$. Then $\widetilde{C}_{\Omega}\le K\sigma^{-1}$, where
\begin{align*}
K&=f(x^{0})-\inf f+\breve{\mathcal{M}}(x^{0})+0.5\breve{\mathcal{H}}_{(2)}(2^{2m-2}\mathrm{R}_{1}(m+1+d)^{\frac{m+1}{2}}+\mathrm{R}_{2}\sqrt{d})\\
&~~~+2^{2m-3}\breve{\mathcal{H}}_{(2)}\mathrm{R}_{1}\sqrt{d}+2^{m-1}\mathrm{R}_{1}\breve{\mathcal{H}}_{(m+2)}(m+2)^{-1}\sqrt{d},
\end{align*}
$\breve{\mathcal{M}}(\cdot)$ is defined as \eqref{breveM}, $\breve{\mathcal{H}}_{(\cdot)}$ is defined as \eqref{breveH} and $\widetilde{C}_{\Omega}$ is given in \eqref{cdd}.
\end{lemma}
\begin{proof}
First, from the definitions of $\mathcal{A}$ and $\mathcal{B}$, we have that
 \begin{align*}
  (\mathcal{A}+\mathcal{B})\gamma^{2}&=\Big[2^{2m-2}\mathrm{R}_{1}\sigma^{m-1}(m+1+d)^{\frac{m+1}{2}}+\sigma^{-1}\mathrm{R}_{2}\sqrt{d}+2^{2m-2}\sigma^{-1}\mathrm{R}_{1}\sqrt{d}\Big]\sigma^{2}\!\!\!\!\!\!\!\!\\
  &\le2^{2m-2}\mathrm{R}_{1}(m+1+d)^{\frac{m+1}{2}}+\mathrm{R}_{2}\sqrt{d}+2^{2m-2}\mathrm{R}_{1}\sqrt{d}.
 \end{align*}
 Next, one can deduce upon invoking the definition of $\mathcal{C}$ that
 \begin{align*}
  \mathcal{C}\gamma^{m+2}=2^{m-1}\sigma^{-1}\mathrm{R}_{1}\sqrt{d}\sigma^{m+2}\le2^{m-1}\mathrm{R}_{1}\sqrt{d},
 \end{align*}
which yields the desired conclusion upon invoking the definition of $\widetilde{C}_{\Omega}$ and \eqref{boundbound}.
\end{proof}
\begin{theorem}[Complexity bound for approximate Goldstein stationarity]\label{comforun}
Consider problem~\eqref{problem}, where $f\in \spb$ with parameters ${\rm R}_1$, ${\rm R}_2$ and $m$ as in \eqref{definitioneq01}. Let $\inf_{u\in \R^d} f(u)>-\infty$, $\delta\in(0,1)$, ${\cal P}$ be defined in \eqref{PPP}, and define
\begin{align}
  \breve{\mathcal{N}}(m)&=
\begin{cases}
 \max\{\min\{5\mathrm{R}_{2},1\}^{-\frac1{\min\{m,d\}}}, \kappa_{1}\}^{4\min\{m,d\} + 2}+1& {\rm if}\ m \ge 1, \\
 \max\{\min\{5\mathrm{R}_{2},1\}^{-\frac1d}, \kappa_{2}\}^{4d+2}+1 & {\rm if} \ m \!=\!0,
 \end{cases}\!\!\!\!\!
\end{align}
where
\begin{align}
\kappa_{1}&=\min\Big\{[2^{m+1}{\rm R_1}(m+d)^{\frac{m}{2}}]^{-\frac{1}{m}},\frac{\delta\mathcal{P}^{-1/d}}{\sqrt{d\pi e}}\Big\}, \ \ \kappa_{2}=\frac{\delta\mathcal{P}^{-1/d}}{\sqrt{d\pi e}}.\label{kappade}
\end{align}
Let $T\ge \breve{\mathcal{N}}(m)$ be a positive integer and $\gamma = \sigma=\breve{\sigma}(m)$, where
 \[
  \breve{\sigma}(m):=
\begin{cases}
 \kappa_{1}T^{-\frac{1}{4\min\{m,d\}+2}}& {\rm if}\ m \ge 1, \\
 \kappa_{2}T^{-\frac{1}{4d+2}} & {\rm if} \ m =0.
 \end{cases}
 \]
 Let {\color{rerevise}$\tau_{k}={\gamma}/{\sqrt{T+1}}$} for $k=0,\cdots,T$.
 Then the following statements hold.
 \begin{enumerate}[label={\rm (\roman*)}]
 \item\label{r120301} For $m\ge1$, under the conditions of Corollary \ref{coro082801}, the sequence $\{x^k\}$ generated by Algorithm~\ref{algorithm 01} satisfies that
\begin{align}
  \min_{0\leq k\leq T}\mathbb{E}\big[\mathrm{dist}(0,\partial_{G}^{\delta}f(x^{k}))^{\frac{1}{2\lceil m/2\rceil}}\big]\!\le\! \frac{\Big(1+\sqrt{\widetilde{K}_{\Omega}}\Big)(2K^{\frac{1}{2}}\kappa_{1}^{-\frac{1}{2}}+2)^{\frac{1}{2\lceil m/2\rceil}}}
  {T^{(\frac{1}{4}-\frac{1}{8\min\{m,d\}+4})\frac{1}{2\lceil m/2\rceil}}},
\end{align}
where $K$ is given in Lemma~\ref{bound1130} and $\widetilde{K}_{\Omega}=8\mu^{-1}K+4d+2\sup_{w\in \mathcal{S}}\|w\|^2$.
 \item\label{r120302} For $m=0$, the sequence $\{x^k\}$ generated by Algorithm~\ref{algorithm 01} satisfies that
 \begin{align*}
 \min_{0\leq k\leq T}\mathbb{E}\left[\mathrm{dist}(0,\partial_{G}^{\delta}f(x^{k}))\right]\le (2K^{\frac{1}{2}}\kappa_{2}^{-\frac{1}{2}}+2)T^{-(\frac{1}{4}-\frac{1}{8d+4})},
\end{align*}
where $K$ is given in Lemma~\ref{bound1130}.
 \end{enumerate}
\end{theorem}
\begin{proof}
Since $T\ge \breve{\cal N}(m)$, we see that $\breve{\sigma}(m)\le 1$, which means $\gamma \in (0,1]$. We can now deduce from Theorem \ref{HBcomplexity01} and $\tau_{k}\equiv \gamma/{\sqrt{T+1}}$ that
$\widetilde{w}_{k_*}^{2} \le {\widetilde{C}_{{\color{revise} \Omega}}}/{\sqrt{T}}$,
where $\widetilde{C}_{{\color{revise} \Omega}}$ is defined in \eqref{cdd}, which implies
\begin{align}
 \mathbb{E}\bigg[\bigg\|\frac{\nabla f_{\sigma}(x^{k_{*}})}{1+\|x^{k_{*}}\|^{m}}\bigg\|^{2}\bigg]\le \frac{\widetilde{C}_{{\color{revise} \Omega}}}{\sqrt{T}},\ \ \text{and hence}\ \ \mathbb{E}\bigg[\frac{\|\nabla f_{\sigma}(x^{k_{*}})\|}{1+\|x^{k_{*}}\|^{m}}\bigg]\le\widetilde{C}_{{\color{revise} \Omega}}^{\frac{1}{2}}T^{-\frac{1}{4}}.\label{eq0828001pm}
\end{align}
On the other hand, if we define
\begin{align*}
 \breve\epsilon(m)=
\begin{cases}
 T^{-\frac{\min\{m,d\}}{4\min\{m,d\}+2}}& {\rm if}\ m \ge 1,\\
 T^{-\frac{d}{4d+2}} & \text{if } m =0,
 \end{cases}
\end{align*}
then for the $T\ge \breve{\mathcal{N}}(m)$, $\delta\in(0,1)$, $\epsilon=\breve{\epsilon}(m)$ and $\sigma=\breve{\sigma}(m)$, one has $0<\epsilon < \min\left\{5{\rm R_2},1\right\}$ and $\sigma$ satisfies \eqref{091001}. Consequently, by Remark~\ref{sdde}, one has
\[
\nabla f_{\sigma}(x^{k_{*}})\in\partial_{G}^{\delta}f(x^{k_{*}})+(1+\|x^{k_{*}}\|^{m})\epsilon\cdot\mathbb{B}.
\]
This implies that
\begin{align*}
 \mathrm{dist}(0,\partial_{G}^{\delta}f(x^{k_{*}}))\le\|\nabla f_{\sigma}(x^{k_{*}})\|+(1+\|x^{k_{*}}\|^{m})\epsilon.
\end{align*}
Rearranging the above display and then taking expectation, one can obtain
\begin{align*}
  \mathbb{E}\left[\frac{\mathrm{dist}(0,\partial_{G}^{\delta}f(x^{k_{*}}))}{1+\|x^{k_{*}}\|^{m}}\right] \le\mathbb{E}\left[\frac{\|\nabla f_{\sigma}(x^{k_{*}})\|}{1+\|x^{k_{*}}\|^{m}}\right]+\epsilon.
\end{align*}
This together with \eqref{eq0828001pm} implies that
\begin{align}
 \mathbb{E}\left[\frac{\mathrm{dist}(0,\partial_{G}^{\delta}f(x^{k_{*}}))}{1+\|x^{k_{*}}\|^{m}}\right] \le\widetilde{C}_{\Omega}^{\frac{1}{2}}T^{-\frac{1}{4}}+\epsilon.\label{eq0828002pm}
\end{align}
We now prove~\ref{r120301}.  Notice from Corollary~\ref{coro082801} that
\begin{align}
 \mathbb{E}[\|x^{k_{*}}\|^{2}]\le \widetilde{M}_{\Omega},\label{eq0828003pm}
\end{align}
where $\widetilde{M}_{\Omega}$ is defined as in \eqref{mdd}. Combining \eqref{eq0828003pm}, \eqref{eq0828002pm} and Lemma \ref{lem082701}, one has
\begin{align}
 \mathbb{E}\big[\mathrm{dist}(0,\partial_{G}^{\delta}f(x^{k_{*}}))^{\frac{1}{2\lceil m/2\rceil}}\big]\le\Big(1+\sqrt{\widetilde{M}_{\Omega}}\Big)\big(2\widetilde{C}_{\Omega}^{\frac{1}{2}}T^{-\frac{1}{4}}+2\epsilon\big)^{\frac{1}{2\lceil m/2\rceil}}.\label{eq113001}
\end{align}
Next, we utilize Lemma~\ref{bound1130} and the above display to obtain the desired result for $m\ge1$. To this end, notice that $\sigma=\breve{\sigma}(m)\in(0,1]$ and $\gamma=\sigma$. Combining Lemma~\ref{bound1130} and the definitions of $\widetilde{M}_{\Omega}$ and $\widetilde{K}_{\Omega}$, one has
\begin{align*}
  \widetilde{M}_{\Omega}\le8\mu^{-1}[\gamma \widetilde{C}_{{\color{revise} \Omega}} + 0.5 \mu d]+2\sup_{w\in \mathcal{S}}\|w\|^2\le8\mu^{-1}(K + 0.5 \mu d)+2\sup_{w\in \mathcal{S}}\|w\|^2=\widetilde{K}_{\Omega}.
\end{align*}
Thus, for $m\ge1$, Lemma~\ref{bound1130} and \eqref{eq113001} yield that
\begin{align*}
 &\mathbb{E}\big[\mathrm{dist}(0,\partial_{G}^{\delta}f(x^{k_{*}}))^{\frac{1}{2\lceil m/2\rceil}}\big]\le\Big(1+\sqrt{\widetilde{K}_{\Omega}}\Big)\big(2K^{\frac{1}{2}}\sigma^{-\frac{1}{2}}T^{-\frac{1}{4}}+2\epsilon\big)^{\frac{1}{2\lceil m/2\rceil}}\\
  &=\Big(1+\sqrt{\widetilde{K}_{\Omega}}\Big)\big[2K^{\frac{1}{2}}(\kappa_{1}T^{-\frac{1}{4\min\{m,d\}+2}})^{-\frac{1}{2}}T^{-\frac{1}{4}}\!+\!2T^{-\frac{\min\{m,d\}}{4\min\{m,d\}+2}}\big]^{\frac{1}{2\lceil m/2\rceil}}\\
  &=\Big(1+\sqrt{\widetilde{K}_{\Omega}}\Big)(2K^{\frac{1}{2}}\kappa_{1}^{-\frac{1}{2}}+2)^{\frac{1}{2\lceil m/2\rceil}}T^{-(\frac{1}{4}-\frac{1}{8\min\{m,d\}+4})\frac{1}{2\lceil m/2\rceil}}.
\end{align*}
Finally, to prove \ref{r120302}, we deduce from Lemma~\ref{bound1130} and \eqref{eq0828002pm} that
\begin{align*}
&\mathbb{E}\left[\mathrm{dist}(0,\partial_{G}^{\delta}f(x^{k_{*}}))\right]\le2K^{\frac{1}{2}}\sigma^{-\frac{1}{2}}T^{-\frac{1}{4}}+2\epsilon\\
  &=2K^{\frac{1}{2}}(\kappa_{2}T^{-\frac{1}{4d+2}})^{-\frac{1}{2}}T^{-\frac{1}{4}}\!+\!2T^{-\frac{d}{4d+2}}=(2K^{\frac{1}{2}}\kappa_{2}^{-\frac{1}{2}}+2)T^{-(\frac{1}{4}-\frac{1}{8d+4})}.
\end{align*}
\end{proof}
}
{\color{revise}\begin{remark}[Explicit bound on $\breve{\mathcal{N}}(m)$]
Note that the result in Theorem~\ref{comforun} requires that $T\ge \breve{\mathcal{N}}(m)$. Here, we derive simpler bounds for $\breve{\mathcal{N}}(m)$ that are \emph{independent} of $d$ when $d$ is large, under the assumptions of Theorem~\ref{comforun}. We first consider the case $m \ge 1$. Assume in addition that $d\ge m$. Then,
\begin{align*}
\kappa_{1}\overset{\mathrm{(a)}}\le(2^{m+1}\mathrm{R}_{1})^{-\frac{1}{m}}(m+d)^{-\frac{1}{2}}\le\mathrm{R}_{1}^{-\frac{1}{m}},
\end{align*}
where we used the definition of $\kappa_{1}$ (see \eqref{kappade}) in (a). From the above display and the definition of $\breve{\mathcal{N}}(m)$, we obtain
\begin{align*}
 \breve{\mathcal{N}}(m)\le \max\{\min\{5\mathrm{R}_{2},1\}^{-1}, \mathrm{R}_{1}^{-1}\}^{\frac{4m+2}{m}}+1.
\end{align*}
Next, in the case $m=0$, let $d\ge\max\{2, 2^{-2}\mathrm{R}_{2}^{-1}\}$. Since $d\ge 2$, we have
\begin{align}\label{eq120102}
\min\{5\mathrm{R}_{2}, 1\}^{-\frac{4d+2}{d}}\le\min\{5\mathrm{R}_{2}, 1\}^{-5}.
\end{align}
 On the other hand, we know
\begin{align}\label{eq120103}
 \kappa_{2}^{4d+2}\overset{\mathrm{(a)}}=\delta^{4d+2}\frac{(4\mathrm{R}_{2})^{-\frac{4d+2}{d}}}{(\sqrt{d\pi e})^{4d+2}}\overset{\mathrm{(b)}}\le \frac{1}{(4\mathrm{R}_{2})^{\frac{4d+2}{d}}d^{2d+1}}\!\overset{\mathrm{(c)}}\le\!\frac{1}{(4\mathrm{R}_{2}d)^{\frac{4d+2}{d}}}
 \overset{\mathrm{(d)}}\le1,
\end{align}
where (a) follows from the definition of $\kappa_{2}$ in \eqref{kappade}, (b) holds because $\sqrt{\pi e} > 1$ and $\delta \in (0,1)$, (c) follows from $d\ge2$ and (d) follows from the fact $d\ge2^{-2}\mathrm{R}_{2}^{-1}$. Now, combining \eqref{eq120102}, \eqref{eq120103} and the definition of $\breve{\mathcal{N}}(m)$, one has
\begin{align*}
 \breve{\mathcal{N}}(m)\le \min\{5\mathrm{R}_{2},1\}^{-5}+1.
\end{align*}
\end{remark}}
{\bf Acknowledgements}. The authors would like to thank Wenqing Ouyang for
discussions on the tail bound and tempered distribution used in section~\ref{sec3orig}. {\color{revise} We would also like to thank the referees for their comments and suggestions, which significantly improved the manuscript}.

\end{document}